\documentclass[11pt]{article}
\usepackage{fancyhdr,amsmath, graphicx, xcolor, amsfonts, layout, mathrsfs, amssymb, amsthm, wrapfig, manfnt}
\usepackage{wasysym}
\usepackage[pagebackref=true,colorlinks,allcolors=blue]{hyperref}
\usepackage{scalefnt}
\usepackage[normalem]{ulem}

\usepackage{pgf}
\usepackage{tikz}
\usepackage{tikz-qtree}
\usepackage{forest}
\usetikzlibrary{shadows,trees,arrows}
\usetikzlibrary{shapes,positioning}

\usepackage{natbib}
\usepackage{verbatim}
\usepackage{array}

\usepackage{mathrsfs}

\newtheorem{theo}{Theorem}
\newtheorem*{theo*}{Theorem}
\newtheorem{defi}{Definition}
\newtheorem{lem}{Lemma}
\newtheorem{cor}{Corollary}
\newtheorem{pro}{Proposition}
\newtheorem{exa}{Example}
\newtheorem{conj}{Conjecture}

\newtheoremstyle{remarques}{\medskipamount}{\medskipamount}{}
                        {0pt}{\bfseries}{.}{ }{}
\theoremstyle{remarques}
\newtheorem{rem}{Remark}

\DeclareMathOperator{\Espe}{\bf E}
\renewcommand{\geq}{\geqslant}
\DeclareMathOperator{\indic}{{\bf 1}}
\renewcommand{\le}{\leqslant}
\renewcommand{\leq}{\leqslant}
\DeclareMathOperator{\Proba}{\bf P}

\DeclareMathOperator{\rpref}{pref}
\DeclareMathOperator{\lpref}{pref}

\newcounter{compteurExo}[section]

\DeclareMathOperator{\casc}{casc}

\newcommand{\g}[1]{\mathbb #1}

\newcommand{\rond}[1]{{\mathscr #1}}

\newcommand{\proba}{\textbf{P}}

\newcommand{\Nfin}{\color{black}}

\usepackage[draft]{fixme}
\fxsetup{layout=pdfcnote}
\FXRegisterAuthor{b}{ab}{B}
\FXRegisterAuthor{pe}{ape}{P}
\FXRegisterAuthor{f}{af}{F}
\FXRegisterAuthor{n}{an}{N}

\begin{document}
 \title{Characterization of stationary probability measures for\\Variable Length Markov Chains}
 \author{Peggy C\'enac\thanks{Universit\'e de Bourgogne,
Institut de Math\'ematiques de Bourgogne,
IMB UMR 5584 CNRS,
9 rue Alain Savary - BP 47870, 21078 DIJON CEDEX, France.
}, Brigitte Chauvin\thanks{Laboratoire de Math\'ematiques de Versailles, UVSQ, CNRS, 
Universit\'e Paris-Saclay, 78035 Versailles, France
}, Fr\'ed\'eric Paccaut\thanks{LAMFA,
CNRS, UMR 7352,
Universit\'e  de Picardie Jules Verne,
33 rue Saint-Leu, 80039 Amiens, France.
} \ 
and Nicolas Pouyanne\footnotemark[2]}
\maketitle
\begin{abstract}
	By introducing a key combinatorial structure for words produced by a Variable Length Markov Chain (VLMC), the \emph{longest internal suffix}, precise characterizations of existence and uniqueness of a stationary probability measure for a VLMC chain are given. These characterizations turn into necessary and sufficient conditions for VLMC associated to a subclass of probabilised context trees: the \emph{shift-stable} context trees. As a by-product, we prove that a VLMC chain whose stabilized context tree is again a context tree has at most one stationary probability measure.
\end{abstract}


\tableofcontents
\section{Introduction}
\label{sec:intro}

Infinite random sequences of letters can be viewed as stochastic chains or as strings produced by a source, in the sense of information theory. In this last frame, context tree models have been introduced by \cite{rissanen/83} as a parsimonious generalization of Markov models to perform data compression. They have been successfully studied and used since then in many fields of applications of probability, including bioinformatics, universal coding, statistics or linguistics.

Statistical use of context tree models requires the possibility of constructing efficient estimators of context trees. For this matter, 
various algorithms (for example Rissanen's \emph{Context} algorithm, or the so-called ``context tree weighting'' \cite{willems/shtarkov/tjalkens/95})
have been developed. The necessity of considering \emph{infinite depth} of context trees has emerged, even for finite memory Markov processes (see for instance \cite{willems/98}).
Moreover, estimators for not necessarily finite memory processes lead to consider infinite context trees. And
estimating context trees is harder (namely the consistency is no more ensured), when the depth of the context tree is infinite. This problem in addressed in \cite{csiszar/talata/06} and \cite{talata/13}.

In biology, persistent random walks are one possible model to address the question of anomalous diffusions in cells (see for instance \cite{FedTanZub}). Actually, such random walks are non Markovian and the displacements and the jumping times are correlated. As pointed in 
\cite{cenac/chauvin/herrmann/vallois/13, CDLO, cenac:hal-01658494},
persistent random walks can be viewed as VLMC for an infinite context tree.

\vskip 5pt
Variable length Markov chains are also a particular case of processes defined by a $g$-function (where the $g$-function is piecewise constant on a countable set of cylinders), also called "cha\^ines \`a liaisons compl\`etes" after \cite{doeblin/fortet/37} or "chains with infinite order" after \cite{harris/55}. Stationary probability measures for VLMC are $g$-measures. The question of uniqueness of $g$-measures has been adressed by many authors when the function $g$ is continuous (in this case, the existence is straightforward), see \cite{johansson/oberg/03}, \cite{fernandez/maillard/05}. Recently, interest raised also for the question of existence and uniqueness when $g$ is not continuous, see \cite{gallo/11}, \cite{gallo/garcia/13}, \cite{desantis/piccioni/12} for a perfect simulation point of view and the more ergodic theory flavoured \cite{gallo/paccaut/13}.

 \vskip 5pt
In this paper we go towards some necessary and sufficient conditions to ensure existence and uniqueness of a stationary probability measure for a general VLMC.
In this introduction we give the thread of the story and the main results. Key concepts are precisely defined in Section \ref{sec:def} while Section \ref{sec:general} is devoted to stating and proving the main theorem (Theorem \ref{fQbij}). The particular case of so-called stable context trees is detailed in Section \ref{sec:stable}, giving a NSC (Theorem~\ref{th:stable}). Several illustrating examples are given in Section \ref{sec:examples}.

\vskip 5pt
Let us recall briefly the probabilistic presentation of Variable Length Markov Chains (VLMC), following \cite{cenac/chauvin/paccaut/pouyanne/12}. Introduce the set $\rond L$ of left-infinite words on the alphabet $\rond A=\{0,1\}$  and  consider a saturated tree $\rond T$ on this alphabet, \emph{i.e.}\@ a tree such that each node has $0$ or $2$ children, whose leaves are words (possibly infinite) on $\rond A$. 
The set of leaves, denoted by $\rond C$ is supposed to be at most countable.

To each leaf $c\in\rond C$, called a context, is attached a probability Bernoulli distribution $q_{c}$ on $\rond A$. Endowed with this probabilistic structure, such a tree is named a probabilised context tree. The related VLMC is defined as the Markov chain $(U_n)_{n\geq 0}$ on $\rond L$ whose transitions are given, for $\alpha\in\rond A$, by
\begin{equation}
 \label{eq:def:VLMC}
 \proba(U_{n+1} = U_n\alpha| U_n)=q_{\lpref (\overline{U_n})}(\alpha),
 \end{equation}
 where $\lpref(u)\in\rond C$ is defined as the only prefix of the right-infinite word $u$  appearing as a leaf of the context tree. As usual, the bar denotes mirror word\footnote{
 In the whole paper, words are as usual read from left to right. One exception to this rule: since the process $(U_n)$ of left-infinite words grows by adding letters to the right, all words that are used to describe suffixes of $U_n$ are read from right to left, hence are written with a bar. See also Remark~\ref{rem:bar}.}
and concatenation of words is written without any symbol.

\vskip 5pt
The heuristic is classically as follows: assume that a stationary measure exists for a given VLMC. Describe all properties this measure should have, then try to prove these properties are sufficient to get the existence of a stationary measure. 

\vskip 5pt
Let $\pi$ be a stationary measure for $(U_n)$. It is entirely defined by its value $\pi(\rond L w)$ on cylinders $\rond L w$, 
for all finite words $w$. By definition \eqref{eq:def:VLMC} of the VLMC, for any letter $\alpha\in\rond A$ and $ w$ non internal word of the context tree,
\begin{equation}
\label{formule:casc0}
\pi(\rond L \overline{w}\alpha) = q_{\lpref ({w})}(\alpha) \pi(\rond L \overline{w}).
\end{equation}
A detailed proof of this formula and the following ones will be given at Lemma \ref{lem:cascade}. This formula applies again for $\pi(\rond L \overline{w})$, and so on, and so forth, until... it is not possible anymore, which means that the suffix of ${w}$ is of the form $\alpha s$ where $\alpha\in\rond A$ and $s$ is an internal word of the context tree. This leads to point out the following decomposition of any finite word $w$:
\[
w=\beta _1\beta _2\dots\beta _{p_w}\alpha _ws_w,
\]
where
 
$\bullet$ $p_w$ is a nonnegative integer and $\beta_i\in \rond A$, for all $i = 1, \dots , p_w$, 

$\bullet$ $s_w$ is the longest internal suffix of $w$,

$\bullet$ $\alpha _w\in\rond A$.

\vskip 5pt
With this decomposition, $s_w$ is called the \emph{lis} of $w$ and $\alpha_w s_w$ the $\alpha$-\emph{lis} of $w$.
Consequently, for any stationary measure $\pi$ and for any finite word $w$, write $ w = v\alpha_w s_w$ where $v$ is a finite word and $\alpha_w s_w$ is the $\alpha$-lis of $ w$ so that
\begin{equation}
\label{formule:casc1}
\pi(\rond L \overline w) = \casc (w) \pi(\rond L \overline{\alpha_w s_w}),
\end{equation}
where $\casc (w)$, 
the \emph{cascade} of $w$ is defined as
\[
\casc (w)=\prod _{1\leq k\leq p_w}q_{\rpref\sigma ^k(w)}(\beta _k).
\]
In the above formula, $\sigma$ is the shift mapping defined by $\sigma\left(\alpha _0\alpha _1\alpha _2\cdots\right)=\alpha _1\alpha _2\cdots$.

\vskip 5pt
Formula \eqref{formule:casc1} indicates that a stationary measure $\pi$ is entirely determined by its value $\pi(\rond L \overline{\alpha v})$ for all $\alpha v$, where $\alpha$ is a letter and $v$ an internal word of the context tree. Further, notice that for any internal word $v$, by disjoint union, using Formula \eqref{formule:casc0},
\begin{equation}
\label{formule:casc2}
\pi\left(\rond L \overline{\alpha v}\right)
=\pi\left(\rond L \overline{ v}\alpha\right)
=\sum _{c\in\rond C,~c=v\cdots}\pi\left(\rond L \overline{c}\right) q_c\left(\alpha\right).
\end{equation}
This means that $\pi$ is in fact determined by the $\pi\left(\rond L \overline c\right)$ where $c$ is a context. Admit for this introduction that only finite contexts matter (Lemma~\ref{lem:pineq0}) and denote by $\rond C^f$ the set of finite contexts. Formula  \eqref{formule:casc1} applies to $\pi\left(\rond L \overline c\right)$:
\[
\pi\left(\rond L \overline c\right) = \casc(c) \pi\left(\rond L \overline{\alpha _cs_c}\right),
\]
so that $\pi$ is entirely determined by its value $\pi\left(\rond L \overline{\alpha s}\right)$ on all the $\alpha s$ which are $\alpha$-lis \emph{of contexts}. Denote by $\rond S$ the set of all the $\alpha$-lis of contexts.
All these values of $\pi$ are connected since by Formula \eqref{formule:casc2}, for any  $\alpha s\in\rond S$,
\[
\pi\left(\rond L \overline{\alpha s}\right) = \sum _{c\in\rond C^f,~c=s\cdots}\casc(\alpha c)\pi\left(\rond L \overline{\alpha _cs_c}\right) = \sum_{\beta t\in\rond S}\pi\left(\rond L \overline{\beta t}\right) \sum _{\substack{c\in\rond C^f\\c=s\cdots\\c=\cdots [\beta t]}}\casc(\alpha c),
\]
where the notation $c=\cdots [\beta t]$ means that $\beta t$ is the $\alpha$-lis of $c$.

Introduce the square matrix $Q=\left( Q_{\alpha s,\beta t}\right) _{(\alpha s,\beta t)\in\rond S^2}$ (at most countable) defined by
\[
Q_{\alpha s,\beta t}
=\sum _{\substack{c\in\rond C^f\\c=t\cdots\\c=\cdots [\alpha s]}}
\casc\left(\beta c\right), 
\]
so that the above formula on $\pi$ writes
\[
\pi\left(\rond L \overline{\alpha s}\right) = \sum_{\beta t\in\rond S}\pi\left(\rond L \overline{\beta t}\right) Q(\beta t, \alpha s).
\]
In otherwords, $\left( \pi\left(\rond L\overline{\alpha s}\right)\right)_{\alpha s\in\rond S}$ is a left-fixed vector of the matrix $Q$. 
It appears that the study of the matrix $Q$ acting on the $\alpha$-lis of contexts is the key tool to characterize a stationary measure for the VLMC. It allows us to prove in Section \ref{sec:general} the following theorem.

\begin{theo*}
\label{th:intro}
Let $(\rond T,q)$ be a probabilised context tree and $U$ the associated VLMC.
Assume that $\forall \alpha\in\rond A$, $\forall c\in\rond C$, $q_c(\alpha)\neq 0$.

\vskip 3pt
(i)
Assume that there exists a finite $U$-stationary probability measure $\pi$ on $\rond L$.
Then the cascade series $\displaystyle\sum _{c\in\rond C^f,~c=\cdots [\alpha s]}\casc (c)$ converge.
Calling $\kappa _{\alpha s}$ its sum,
\begin{equation}
\sum _{\alpha s\in\rond S}\pi\left(\rond L\overline{\alpha s}\right)
\kappa _{\alpha s}=1.
\end{equation}

\vskip 3pt
(ii)
Assume that the cascade series $\displaystyle\sum _{c\in\rond C^f,~c=\cdots [\alpha s]}\casc (c)$ converge.
Then, there is a bijection between the set of $U$-stationary probability
measures on $\rond L$ and the set of left-fixed vectors $\left( v_{\alpha s}\right)_{\alpha s\in\rond S}$ of $Q$ that satisfy
\begin{equation}
\sum _{\alpha s\in\rond S}v_{\alpha s}\kappa _{\alpha s}=1.
\end{equation}
\end{theo*}
The characterization given in this theorem is expressed via the probability distributions $q_c$. Nevertheless,
the role of context lis and $\alpha$-lis suggests that the \emph{shape} of the context tree matters a lot. Actually, in the case of \emph{stable} trees (so called because they are stable by the shift), it turns out that the matrix~$Q$ is stochastic and irreducible. Consequently, the theorem above provides in Theorem~\ref{th:stable} a NSC in terms of the recurrence of~$Q$.
As a corollary, for context trees $(\rond T,q)$ whose stabilized (the smallest stable context tree that contains $\rond T$)
has again at most countably many infinite branches, the associated chain admits either a unique or no stationary probability measure.

Moreover, in the particular case of a \emph{finite} number of context $\alpha$-lis, Theorem~\ref{cor:finite} gives a rather easy condition for the existence and uniqueness of a stationary probability measure: there exists a unique probability measure if and only if the cascade series converge. 
Section~\ref{sec:stable}, which deals with stable trees, is made complete with the link that can be highlighted with the semi-Markov chains theory.

Section~\ref{sec:examples} contains a list of examples that illustrate different configurations of infinite branches and context $\alpha$-lis, and also the respective roles of the geometry of the context tree on the one hand and on the other the asymptotic behaviour of the Bernoulli distributions $q_c$ when the size of the context $c$ (having a given $\alpha$-lis) grows to infinity.
The final Section \ref{sec:towards} gives some tracks towards better characterizations in the non-stable case. A conjecture is set: if the set of infinite branches does not contain any shift-stable subset, then there exists a unique stationary probability measure for the associated VLMC.
\section{Definitions}
\label{sec:def}

\subsection{VLMC}

In the whole paper, $\rond A=\{ 0,1\}$ denotes the set of \emph{letters} (the \emph{alphabet}) and $\rond L$ and $\rond R$ respectively denote the left-infinite and the right-infinite $\rond A$-valued sequences\footnote{$\g N$ denotes the set of natural integers $\{ 0,1,2,\cdots\}$.}:
\[
\rond L=\rond A^{-\g N}
{\rm ~~and~~}
\rond R=\rond A^{\g N}.
\]
The set of finite words, sometimes denoted by $\rond A^*$, will be denoted by $\rond W$:
\[
\rond W=\bigcup _{n\in\g N}\rond A^n,
\]
the set $\rond A^0$ being the emptyset.
When $\ell\in\rond L$,  $r\in\rond R$ and $v,w\in\rond W$, the concatenations of $\ell$ and $w$, $v$ and $w$, $w$ and $r$ are respectively denoted by $\ell w$, $vw$ and $wr$.
More over, the finite word $w$ being given,
\[
\rond Lw
\]
denotes the cylinder made of left-infinite words having $w$ as a suffix.
Finally, when $w=\alpha _1\cdots\alpha _d\in\rond W$, $\ell =\cdots\alpha _{-2}\alpha _{-1}\alpha _{0}\in\rond L$ or $r=\alpha _0\alpha _1\alpha _2\cdots\in\rond R$, the bar denotes the mirror object :
\begin{align*}
&\overline w=\alpha _d\alpha _{d-1}\cdots\alpha _{1}\in\rond W,\\
&\overline\ell =\alpha _{0}\alpha _{-1}\alpha _{-2}\cdots\in\rond R,\\
&\overline r=\cdots\alpha _{2}\alpha _{1}\alpha _{0}\in\rond L.
\end{align*}

A VLMC is an $\rond L$-valued Markov chain, defined by a so-called \emph{probabilised context tree}.
We give here a compact description.
One can refer to~\cite{cenac/chauvin/paccaut/pouyanne/12} for an extensive definition.

A \emph{context tree} is a rooted binary tree $\rond T$ which has an \emph{at most countable set of infinite branches};
an infinite sequence $r\in\rond R$ is an \emph{infinite branch} of $\rond T$ whenever all its finite prefixes belong to $\rond T$.
As usual, the nodes of the tree are canonically labelled by words on $\rond A$.
In the example of Figure~\ref{fig:ex0}, the tree has two infinite branches: $(01)^\infty$ and $1^\infty$.
A node of a context tree $\rond T$ will be called a \emph{context} when it is a finite leaf or an infinite branch of $\rond T$.
The sets of all contexts, finite leaves and infinite branches are respectively denoted by
\[
\rond C,~\rond C^f
{\rm ~and~}
\rond C^i.
\]
They are all at most countable sets.
A finite word $w\in\rond W$ will be called an \emph{internal node} when it is \emph{strictly} internal as a node of $\rond T$;
it will be called \emph{nonexternal} whenever it is internal or a context.
In the same vein a finite word or a right-infinite sequence will be said \emph{external} when it is strictly external and \emph{noninternal} when it is external or a context.
The sets of internal words
is denoted by
\[
\rond I
.
\]
\begin{defi}[$\rpref$ of a noninternal word]
Let $\rond T$ be a context tree and $w$ be a noninternal finite or right-infinite word.
Then, we denote by $\rpref (w)$ the unique prefix of $w$ which is a context of $\rond T$.
\end{defi}

\begin{figure}
\begin{center}
\begin{tikzpicture}[scale=0.6]
\tikzset{every leaf node/.style={circle,fill,scale=0.5},every internal node/.style={circle,fill,circle,scale=0.01}}
\Tree
[.{} 
\edge[line width=2pt,red];[.{} 
{}; 
\edge[line width=2pt,red];[.{} 
\edge[line width=2pt,red];[.{} 
{}; 
\edge[line width=2pt,red];[.{} 
[.{} 
{}; 
[.{} 
[.{} 
{}; 
[.{} 
\edge[line width=2pt,dashed];\node[fill=white,draw=white]{};
{}; 
]
]
{}; 
]
]
\edge[line width=2pt,red];\node{};
]
]
{}; 
]
]
[.{} 
{}; 
[.{} 
{}; 
[.{} 
{}; 
[.{} 
{}; 
[.{} 
{}; 
[.{} 
{}; 
[.{} 
{}; 
[.{} 
{}; 
\edge[line width=2pt,dashed];\node[fill=white,draw=white]{};
]
]
]
]
]
]
]
]
]
\draw (0,0.5) node{$\emptyset$};
\draw (-1.7,-0.8) node{$0$};
\draw (1.6,-0.8) node{$1$};
\draw (-2.9,-2.4) node{$00$};
\draw (2.2,-2) node{$11$};
\draw (-2.9,-4.4) node{$0100$};
\draw (2.5,-8.5) node{$1^70$};
\draw (4.5,-9.5) node{$1^\infty$};
\draw (-1.5,-9.8) node{$(01)^\infty$};
\draw (0.5,-5.5) node{$c$};
\draw (10,-5) node{$c=01011=\rpref ({\color{red}01011}1101000\cdots)$};
\end{tikzpicture}
\end{center}
\caption{an example of context tree.
It has two infinite branches: $1^\infty$ and $(01)^\infty$.}
\label{fig:ex0}
\end{figure}
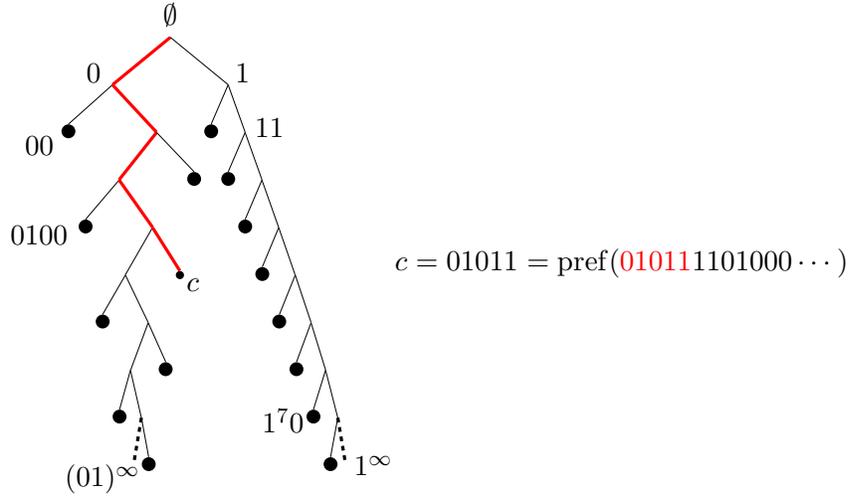

\begin{defi}[shift mapping]
The shift mapping $\sigma:\rond R\to\rond R$ is defined by $\sigma\left(\alpha _0\alpha _1\alpha _2\cdots\right)=\alpha _1\alpha _2\cdots$.
The definition is extended to finite words (with $\sigma(\emptyset)=\emptyset$).
\end{defi}

A \emph{probablised context tree} is a context tree $\rond T$ endowed with a family of probability measures $\left( q_c\right) _{c\in\rond C}$ on $\rond A$ indexed by the (finite or infinite) contexts of $\rond T$.
To any probabilised context tree, one can associate a VLMC, which is the $\rond L$-valued Markov chain $\left( U_n\right)_{n\geq 0}$ defined by its transition probabilities given by
\begin{equation}
\forall n\geq 0,~ \proba\left( U_{n+1}=U_n\alpha |U_n\right)
=q_{\rpref\left(\overline{U_n}\right)}(\alpha ).
\end{equation}

Note that if one denotes by $X_n$ the rightmost letter of the sequence $U_n\in\rond L$ so that
\[
\forall n\geq 0,~U_{n+1}=U_nX_{n+1},
\]
then, when the context tree has at least one infinite context, the final letter process $(X_n)_{n\geq 0}$ is generally \emph{not} a Markov process.
When the tree is finite, $(X_n)_{n\geq 0}$ is a usual $\rond A$-valued Markov chain whose order is the height of the tree, \emph{i.e.} the length of its longest branch.
The vocable VLMC is somehow confusing but commonly used.

\begin{defi}[non-nullness]\label{defi:non-null}
	A probabilised context tree $(\rond T,q)$ is \emph{non-null} if for all $c\in\rond C$ and all $\alpha\in\rond A$, $q_c(\alpha)\neq 0$.
\end{defi}

\begin{rem}\label{rem:bar}
It would have been probably simpler to define a VLMC as an $\rond R$-valued Markov chain $\left( U_n\right) _n$ taking as transition probabilities:
$\proba\left( U_{n+1}=\alpha U_n|U_n\right) =q_{\rpref\left({U_n}\right)}(\alpha )$.
With this definition, words would have grown by adding a letter to the left, and everything would have been read from left to right, avoiding the emergence of bars in the text ($\overline w$, $\overline{U_n}$, \dots ).
We decided to adopt our convention ($U_n$ is $\rond L$-valued) in order to fit the already existing literature on VLMC and more generally on stochastic processes.
\end{rem}

\subsection{Cascades, lis and $\alpha$-lis}

\begin{defi}[lis and $\alpha$-lis]\label{def:alphalis}
Let $\rond T$ be a context tree.
If $w\in\rond W$ is a non empty finite word, $w$ can be uniquely written as 
\[
w=\beta _1\beta _2\dots\beta _{p_w}\alpha _ws_w,
\]
where
 
$\bullet$ $p_w\geq 0$ and $\beta_i\in \rond A$, for all $i\in\{ 1, \dots , p_w\}$, 

$\bullet$ $\alpha _w\in\rond A$,

$\bullet$ $s_w$ is the longest internal strict suffix of $w$.

Note that $s_w$ may be empty.
When $p_w=0$, there are no $\beta$'s and $w=\alpha _ws_w$.

Vocabulary:
the longest internal suffix $s_w$ is abbreviated as the \emph{lis} of $w$;
the noninternal suffix $\alpha _ws_w$ is the \emph{$\alpha$-lis} of $w$.
\end{defi}
Any word has an $\alpha$-lis, but we will be mainly interested by the $\alpha$-lis of \emph{contexts}.
The set of $\alpha$-lis \emph{of the finite contexts} of $\rond T$ will be denoted by $\rond S\left(\rond C\right)$, or more shortly by $\rond S$:
\[
\rond S=\left\{ \alpha _cs_c,~c\in\rond C^f\right\};
\]
this is an at most countable set (like $\rond C$).
For any $u,v,w\in\rond W$, the notations
\begin{equation}
\label{notationsPrefLis}
v=u\cdots
{\rm ~~and~}
w=\cdots [u]
\end{equation}
stand respectively for ``$u$ is a prefix of $v$'' and ``$u$ is the $\alpha$-lis of $w$''.

\vskip 10pt
\begin{exa}[computation of a lis]\label{exa:lis}\

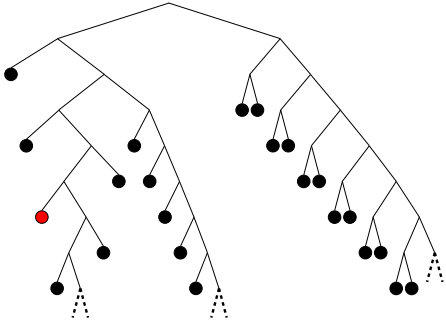
\begin{minipage}{0.4\textwidth}
\centering
\begin{tikzpicture}[scale=0.45]
		\tikzset{every leaf node/.style={draw,circle,fill},every internal node/.style={draw,circle,scale=0.01}}
		\Tree [.{} [.{} {} [.{} [.{} {} [.{} [.{} \node[fill=red,draw=black]{}; [.{} [.{} {} [.{} \edge[line width=2pt,dashed];\node[fill=white,draw=white]{}; \edge[line width=2pt,dashed];\node[fill=white,draw=white]{}; ] ] {} ] ] {} ] ] [.{} {} [.{} {} [.{} {} [.{} {} [.{} {} [.{} \edge[line width=2pt,dashed];\node[fill=white,draw=white]{}; \edge[line width=2pt,dashed];\node[fill=white,draw=white]{}; ] ] ] ] ] ] ] ] [.{} [.{} {} {} ] [.{} [.{} {} {} ] [.{} [.{} {} {} ] [.{} [.{} {} {} ] [.{} [.{} {} {} ] [.{} [.{} {} {} ] [.{} \edge[line width=2pt,dashed];\node[fill=white,draw=white]{}; \edge[line width=2pt,dashed];\node[fill=white,draw=white]{}; ] ] ] ] ] ] ] ]
\end{tikzpicture}
\end{minipage}
\begin{minipage}{0.6\textwidth}
In this example, the finite contexts of the context tree are the following ones (they completely define the context tree):
$(01)^p00$, $(01)^r1$, $01^r0$, $1^q00$, $1^q01$, $p\geq 0$, $q\geq 1$, $r\geq 2$.

Take for example the context $010100$. Remove successively letters from the left until you get an internal word: $10100$ is external, $0100$ is noninternal, $100$ is noninternal, $00$ is noninternal.
The suffix $0$, for the first time, is internal:
this is the lis of $010100$.
The last removed letter is $\alpha=0$ so that the $\alpha$-lis is $00$.

In the array hereunder, the left side column consists in the list of all $\alpha$-lis of the context tree.
For any $\alpha s\in\rond S$, the list of all contexts having $\alpha s$ as an $\alpha$-lis is given in the right side column.

\begin{center}
\begin{tabular}{r|l}
$\alpha s\in\rond S$&contexts having $\alpha s$ as an $\alpha$-lis\\
\hline
$00$&$1^q00$, $(01)^p00$, $p\geq 0$, $q\geq 1$\\
$101$&$1^q01, q\geq 1$\\
$01011$&$(01)^r1$, $r\geq 2$\\
$01^r0$, $r\geq 2$&$01^r0$
\end{tabular}
\end{center}
\end{minipage}
\end{exa}

\begin{defi}[cascade]\label{def:cascade}
Let $\left(\rond T,q\right)$ be a probabilised context tree.
If $w\in\rond W$ writes $w=\beta _1\beta _2\dots\beta _{p}\alpha s$ where $p\geq 0$ and where $\alpha s$ is the $\alpha$-lis of $w$, the \emph{cascade} of $w$ is defined as
\[
\casc (w)=\prod _{1\leq k\leq p}q_{\rpref\sigma ^k(w)}(\beta _k),
\]
where an empty product equals $1$, which occurs if, and only if $w$ equals its own $\alpha$-lis.
The cascade of $\emptyset$ is defined as being $1$.
Note that $\casc (\alpha s) = 1$ for any $\alpha s\in\rond S$.
\end{defi}

In Example \ref{exa:lis}, $\casc(010100)=q_{101}(0)q_{0100}(1)q_{100}(0)q_{00}(1)$.

\begin{rem}
\label{rem:cascRem}
For any $w\in\rond W$, $\casc (w)=\casc (0w)+\casc (1w)$ if, and only if $w$ is noninternal.
Indeed, if $w$ is internal, the sum equals $2$ whereas $\casc (w)\leq 1$.
\end{rem}

\begin{defi}[cascade series]
\label{defi:cascade_series}
For every  $\alpha s\in\rond S$, the \emph{cascade series of $\alpha s$} (related to $(\rond T,q)$) is the at most countable family of cascades of the finite contexts having $\alpha s$ as their $\alpha$-lis.
In other words, with notations~\eqref{notationsPrefLis}, it is the family
\[
\left(\casc (c)\right) _{c\in\rond C^f,~c=\cdots [\alpha s]}.
\]
\end{defi}

Since the cascades are positive numbers, the summability of a family of cascades of a probabilised context tree is equivalent to the convergence of the series associated to any total order on the set of contexts indexing the family.
The assertion
\begin{equation}
\label{convCasc}
\forall \alpha s\in\rond S,~\sum _{c\in\rond C^f,~c=\cdots [\alpha s]}\casc (c)<+\infty
\end{equation}
will be called \emph{convergence of the cascade series}.
When the cascade series converge, $\kappa _{\alpha s}$ denotes the sum of the cascade series relative to $\alpha s\in\rond S$:
\begin{equation}
\label{defKappa}
\kappa _{\alpha s}=\sum _{c\in\rond C^f,~c=\cdots [\alpha s]}\casc (c).
\end{equation}

\begin{rem}
The finiteness of the set $\rond C^i$ of infinite branches on one side, and of the set $\rond S$ of context $\alpha$-lis on the other side are not related.
In Section~\ref{subsubsec:SfiniCIinfini}, one finds an example of context tree for which $\rond S$ is finite while $\rond C^i$ is infinite.
In the example of Section~\ref{subsubsec:SinfiniCIfini}, $\rond S$ is infinite while $\rond C^i$ is finite.
The left-comb of left-combs has infinite $\rond C^i$ and $\rond S$ (see Section~\ref{subsubsec:pgPg}).
Finally, the double bamboo has finite $\rond C^i$ and $\rond S$, see Example~\ref{exa:dbamboo}.
\end{rem}

\subsection{$\alpha$-lis matrix $Q$}
\label{subsec:Q}

For any $(\alpha s,\beta t)\in\rond S^2$, with notations~\eqref{notationsPrefLis}, define
\[
\label{defQ}
Q_{\alpha s,\beta t}
=\sum _{\substack{c\in\rond C^f\\c=t\cdots\\c=\cdots [\alpha s]}}
\casc\left(\beta c\right)
\in [0,+\infty ].
\]
As the set $\rond S$ is at most countable, the family $Q=\left( Q_{\alpha s,\beta t}\right) _{(\alpha s,\beta t)\in\rond S^2}$ will be considered as a matrix, finite or countable, for an arbitrary order.
The convergence of the cascade series of $(\rond T,q)$ is sufficient to ensure the finiteness of the coefficients of $Q$.

\section{Results for a general context tree}
\label{sec:general}

In this section no assumption is made on the shape of the context tree. After two key lemmas, we state and prove the main theorem that establishes precise connections between stationary probability measures of the VLMC and left-fixed vectors of the matrix $Q$ defined in Subsection~\ref{subsec:Q}.

\subsection{Two key lemmas}

\begin{lem}
\label{lem:cascade}
{\bf (Cascade formulae)}

Let $(\rond T,q)$ be a probabilised context tree and $\pi$ be a stationary probability measure for the corresponding VLMC.

(i)
For every noninternal finite word $w$ and for every $\alpha\in\rond A$,
\begin{equation}
\label{cascade1}
\pi\left(\rond L\overline{w}\alpha\right) =q_{\rpref (w)}(\alpha )\pi\left(\rond L\overline{w}\right) .
\end{equation}

(ii)
For every right-infinite word $r\in\rond R$ and for every $\alpha\in\rond A$,
\begin{equation}
\label{cascade2}
\pi\left(\overline{r}\alpha\right) =q_{\rpref (r)}(\alpha )\pi\left(\overline{r}\right) .
\end{equation}

(iii)
For every finite non empty word $w$, if one denotes by $\alpha_ws_w$ the $\alpha$-lis of $w$, then
\begin{equation}
\label{cascade3}
\pi\left(\rond L\overline{w}\right) =\casc(w)\pi\left(\rond L\overline{\alpha_ws_w}\right) .
\end{equation}
\end{lem}

\begin{proof}[Proof of lemma \ref{lem:cascade}]

(i)
Assume first that $\pi\left(\rond L\overline{w}\right)\neq 0$.
Then, since $w$ is noninternal, $\rpref (w)$ is well defined so that, by stationarity,
\begin{align*}
\pi\left(\rond L\overline{w}\alpha\right)
&=\proba _\pi\left( U_1\in\rond L\overline{w}\alpha\right)\\
&=\proba _\pi\left( U_{1}\in\rond L\overline{w}\alpha |U_0\in\rond L\overline w\right)
\proba _\pi\left( U_0\in\rond L\overline{w}\right)\\
&=q_{\rpref\left( w\right)}(\alpha)\proba _\pi\left( U_0\in\rond L\overline{w}\right)
=q_{\rpref\left( w\right)}(\alpha)\pi\left(\rond L\overline{w}\right)
\end{align*}
proving~\eqref{cascade1}.
If $\pi\left(\rond L\overline{w}\right) =0$, then, by stationarity, $\pi\left(\rond L\overline{w}\alpha\right) =\proba _\pi\left( U_1\in\rond L\overline w\alpha\right)\leq\proba _\pi\left( U_0\in\rond L\overline w\right)=0$ so that~\eqref{cascade1} remains true.

(ii)
Since $r$ is infinite, the context $\rpref (r)$ is always defined
(it may be finite or infinite).
Consequently,
\begin{align*}
\pi\left(\overline{r}\alpha\right)
&=\Proba _\pi\left( U_1=\overline{r}\alpha\right)\\
&=\Espe _\pi\Big(\Espe _\pi\left( \indic _{\left\{ U_1=\overline r\alpha\right\}} |U_0\right)\Big)\\
&=\Espe _\pi\left( \indic _{\left\{ U_0=\overline r\right\}}\Proba _\pi\left( U_1=U_0\hskip 1pt\alpha |U_0\right)\right)
=q_{\rpref (r)}(\alpha )\pi\left(\overline{r}\right) .
\end{align*}

(iii)
Direct induction from Formula~\eqref{cascade1}.
\end{proof}

The following lemma ensures that a stationary probability measure weights finite words and only finite words.
\begin{lem}
\label{lem:pineq0}
Let $(\rond T,q)$ be a non-null probabilised context tree.
Assume that $\pi$ is a stationary probability measure
for the associated VLMC. Then 

(i) $\forall w\in\rond W$, $\pi\left(\rond L\overline{w}\right)\neq 0$.

(ii) $\forall r\in\rond R$, $\pi (\overline r)=0$.
\end{lem}

\begin{proof}[Proof of lemma \ref{lem:pineq0}]
(i) We prove that if $w$ is a finite word and if $\alpha\in\rond A$, then $\left[\pi\left(\rond L\overline w\alpha\right) =0\right]\Rightarrow\left[\pi \left(\rond L\overline w\right)=0\right]$.
An induction on the length of $w$ is then sufficient to prove the result since  $\pi (\rond L)=1$.
\begin{enumerate}
\item
Assume that $w\notin\rond I$ and that $\pi\left(\rond L\overline w\alpha\right) =0$.
Then, as a consequence of the cascade formula~\eqref{cascade1}, $0=\pi\left(\rond L\overline w\right)q_{\rpref (w)}(\alpha )$.
As no $q_c$ vanishes, $\pi\left(\rond L\overline w\right) =0$.
\item
Assume now $w\in\rond I$ and $\pi\left(\rond L\overline w\alpha\right) =0$.
Then, by disjoint union and stationarity of $\pi$,
\[
0=\sum _{\substack{c\in\rond C^f\\c=w\dots}}\pi\left(\rond L\overline c\alpha\right)
+\sum _{\substack{c\in\rond C^i\\c=w\dots}}\pi\left(\overline c\alpha\right)
=\sum _{\substack{c\in\rond C^f\\c=w\dots}}\pi\left(\rond L\overline c\right)q_c(\alpha )
+\sum _{\substack{c\in\rond C^i\\c=w\dots}}\pi\left(\overline c\right)q_c(\alpha ).
\]
As no $q_c$ vanishes, all the $\pi\left(\rond L\overline c\right)$ and the $\pi (\overline c)$ necessarily vanish so that, by disjoint union,
\[
\pi\left(\rond L\overline w\right)
=\sum _{\substack{c\in\rond C^f\\c=w\dots}}\pi\left(\rond L\overline c\right)
+\sum _{\substack{c\in\rond C^i\\c=w\dots}}\pi\left(\overline c\right)
=0.
\]
\end{enumerate}

(ii) Denote $r=\alpha _1\alpha _2\cdots$ and $r_n=\alpha _n\alpha _{n+1}\cdots$ its $n$-th suffix, for every $n\geq 1$.
Since $\pi$ is stationary, an elementary induction from Formula~\eqref{cascade2} implies that, for every $m\geq 1$,
\begin{equation}
\label{pir}
\pi\left(\overline r\right)
=\left(\prod _{k=1}^mq_{\rpref (r_{k+1})}\left(\alpha _k\right)\right)\pi\left(\overline{r_{m+1}}\right).
\end{equation}

\begin{enumerate}
\item
Assume first that $r=st^\infty$ is ultimately periodic, where $s$ and $t$ are finite words, $t=\beta _1\cdots\beta _T$ being nonempty.
Then, because of~\eqref{pir}, $\pi (\overline r)\leq\pi\left(\overline t^\infty\right)$ and $\pi\left(\overline t^\infty\right)=\rho\pi\left(\overline t^\infty\right)$ where
\[
\rho=
\prod _{k=1}^T
q_{\rpref (\beta _{k+1}\cdots\beta _Tt^\infty)}\left(\beta _{k}\right).
\]
Since the probability measures $q_c$ are all assumed to be nontrivial, then $0<\rho <1$, which implies that $\pi\left(\overline t^\infty\right)=0$.

\item
Assume on the contrary that $r$ is aperiodic.
Then, $m\neq n\Longrightarrow r_n\neq r_m$ for all $n,m\geq 1$:
the $r_n$ are all distinct among the infinite branches of the context tree.
Thus, by disjoint union, 
\[
\sum _{n\geq 1}\pi\left(\overline{r_n}\right)
\leq\sum _{c\in\rond C^i}\pi\left(\overline c\right)
\leq \pi\left(\rond L\right)=1,
\]
which implies in particular that $\pi\left(\overline{r_n}\right)$ tends to $0$ when $n$ tends to infinity.
Since $\pi\left(\overline r\right)\leq\pi\left(\overline{r_n}\right)$ because of Formula~\eqref{pir}, this leads directly to the result.
\end{enumerate}
\end{proof}

\subsection{Main theorem}

\begin{defi}
Let $A=\left( a_{\ell c}\right)_{(\ell ,c)\in\rond E^2}$ be a matrix with real entries, indexed by a totally ordered set $\rond E$ supposed to be finite or denumerable.
A \emph{left-fixed vector} of $A$ is a row-vector $X=(x_{k})_{k\in\rond E}\in\g R^{\rond E}$, indexed by $\rond E$, such that $XA=X$.
In particular, this implies that the matrix product $XA$ is well defined, which means that for any $c\in\rond E$, the series $\sum _\ell x_\ell a_{\ell ,c}$ is convergent.
Note that, whenever $X$ and $A$ are infinite dimensional and have nonnegative entries, this summability does not depend on the chosen order on the index set $\rond E$.
\end{defi}

\vskip 5pt
Denote by $\rond M_1\left(\rond L\right)$ the set of probability measures on $\rond L$.
\newcommand{\fff}{f}
Define the mapping $\fff$ as follows:
\[
\begin{array}{rcl}
\fff :\rond M_1\left(\rond L\right) & \longrightarrow &[0,1]^{\rond S}
\\
\pi & \longmapsto & \Big( \pi\left(\rond L\overline{\alpha s}\right)\Big) _{\alpha s\in\rond S}.
\end{array}
\]
\begin{theo}
\label{fQbij}
Let $(\rond T,q)$ be a non-null probabilised context tree and $U$ the associated VLMC.

\vskip 3pt
(i)
Assume that there exists a finite $U$-stationary probability measure $\pi$ on $\rond L$.
Then the cascade series~\eqref{convCasc} converge.
Furthermore, using notation~\eqref{defKappa},
\begin{equation}
\label{masseTotale}
\sum _{\alpha s\in\rond S}\pi\left(\rond L\overline{\alpha s}\right)
\kappa _{\alpha s}=1.
\end{equation}

\vskip 3pt
(ii)
Assume that the cascade series~\eqref{convCasc} converge.
Then, $f$ induces a bijection between the set of $U$-stationary probability
measures on $\rond L$ and the set of left-fixed vectors $\left( v_{\alpha s}\right)_{\alpha s\in\rond S}$ of $Q$ which satisfy\begin{equation}
\label{posrec}
\sum _{\alpha s\in\rond S}v_{\alpha s}\kappa _{\alpha s}=1.
\end{equation}
\end{theo}
For an example of application of this theorem, see Section~\ref{subsubsec:SfiniCIinfini}.
\begin{proof}[Proof of theorem~\ref{fQbij}.]~\\
Proof of {\it (i)}.
If $\pi$ is a stationary probability measure, 
disjoint union, Lemma~\ref{lem:pineq0}(i) and the cascade formula~\eqref{cascade3} imply that
\[
1=\sum _{c\in\rond C^f}\pi\left(\rond L\overline c\right)
=\sum _{c\in\rond C^f}\casc (c)\pi\left(\rond L\overline {\alpha _cs_c}\right).
\]
Gathering together all the contexts that have the same $\alpha$-lis leads to
\[
1=\sum _{\alpha s\in\rond S}\pi\left(\rond L\overline{\alpha s}\right)\left(\sum _{c\in\rond C^f,~c=\cdots [\alpha s]}\casc (c)\right).
\]
Now, by Lemma~\ref{lem:pineq0}(ii), $\pi\left(\rond L\overline {\alpha s}\right)\neq 0$ for all $\alpha s\in\rond S$. This forces the sums of cascades to be finite.

\vskip 5pt
Proof of {\it (ii)}.

1) \emph{Injectivity}.
Let $\pi$ be a stationary probability measure on $\rond L$.
As the cylinders based on finite words generate the whole $\sigma$-algebra, $\pi$ is determined by the $\pi\left(\rond L\overline w\right)$, $w\in\rond W$.
Now write any $w\in\rond W\setminus\{\emptyset\}$ as $w=p\alpha s$ where $\alpha s$ is the $\alpha$-lis of $w$ and $p\in\rond W$
(beware, $\alpha s$ may not be the $\alpha$-lis of a context).
As $\pi$ is stationary, the cascade formula~\eqref{cascade3} entails $\pi\left(\rond L\overline w\right) =\casc (w)\pi\left(\rond L\overline{\alpha s}\right)$.
As a consequence $\pi$ is determined by its values on the words $\overline{\alpha s}$ where $ s\in\rond I$ is internal and $\alpha\in\rond A$.
Now, as $ s\in\rond I$, by disjoint union, cascade formula~\eqref{cascade1} and Lemma~\ref{lem:pineq0}(i),
\[
\pi\left(\rond L\overline{\alpha s}\right)
=\pi\left(\rond L\overline{ s}\alpha\right)
=\sum _{c\in\rond C^f,~c= s\cdots}\pi\left(\rond L\overline{c}\right) q_c\left(\alpha\right).
\]
This means that $\pi$ is in fact determined by the $\pi\left(\rond L\overline c\right)$ where $c$ is a finite context.
Lastly, as above, the stationarity of $\pi$, the cascade formula~\eqref{cascade3} and the decomposition of any context $c$ into $c=p_c\alpha _cs_c$ where $\alpha _cs_c$ is the $\alpha$-lis of $c$ together imply that $\pi$ is determined by the $\pi\left(\rond L\overline{\alpha s}\right)$ where $s\in\rond S$
(remember, $\rond S$ denotes the set of all $\alpha$-lis \emph{of contexts}).
We have proved that the restriction of $\fff$ to stationary measures is one-to-one.

\noindent
2) \emph{Image of a stationary probability measure}.
Let $\pi\in\rond M_1\left(\rond L\right)$ be stationary.
By disjoint union, as above, if $\alpha s\in\rond S$,
\[
\pi\left(\rond L\overline{\alpha s}\right)
=\sum _{c\in\rond C^f,~c=s\cdots}\pi\left(\rond L\overline{c}\right) q_c\left(\alpha\right) .
\]
Applying the cascade formula~\eqref{cascade3} to all contexts in the sum and noting that $\casc (\alpha c)=q_c(\alpha )\casc (c)$, one gets
\[
\pi\left(\rond L\overline{\alpha s}\right)
=\sum _{c\in\rond C^f,~c=s\cdots}\casc (\alpha c)\pi\left(\rond L\overline{\alpha _cs_c}\right).
\]
Gathering together all the contexts that have the same $\alpha$-lis entails
\[
\pi\left(\rond L\overline{\alpha s}\right)
=\sum _{\beta t\in\rond S}\pi\left(\rond L\overline{\beta t}\right)\left( \sum _{c\in\rond C^f,~c=s\cdots =\cdots [\beta t]}\casc (\alpha c)\right).
\]
This means that the row vector $\left(\pi\left(\overline{\alpha s}\right)\right) _{\alpha s\in\rond S}$ is a left-fixed vector for the matrix $Q$.
We have shown that $\fff$ sends a stationary probability measure to a left-fixed vector for~$Q$.
Moreover, as in the proof of (i), Equality~\eqref{masseTotale} holds.

\noindent
3) \emph{Surjectivity}.

Let $\left(v_{\alpha s}\right) _{\alpha s\in\rond S}\in [0,1]^{\rond S}$ be a row vector, left-fixed by~$Q$, that satisfies $\sum _{\alpha s\in\rond S}v_{\alpha s}\kappa _{\alpha s}=1$.
Let $\mu$ be the function defined on $\rond S$ by $\mu\left( \alpha s\right)=v_{\alpha s}$.
Denoting by $\alpha _cs_c$ the $\alpha$-lis of a context $c$, $\mu$ extends to any finite nonempty word in the following way:
\begin{equation}
\label{prolongementV}
\forall w\in\rond W\setminus\{\emptyset\},~\mu (w)=\casc (w)\sum _{c\in\rond C^f,~c=s_w\cdots}\casc (\alpha_wc)\mu\left(\alpha _cs_c\right)\in [0,+\infty ].
\end{equation}
Notice that this definition actually extends $\mu$ because of the fixed vector property, and that, at this moment of the proof, $\mu (w)$ might be infinite.
Notice also that this implies $\mu(w)=\casc(w)\mu(\alpha_ws_w)$ for any $w\in\rond W$, $w\neq\emptyset$.

For every $n\geq 1$ and for all $w\in\rond W$ such that $|w|=n$, define $\pi_n\left(\overline w\right) =\mu\left( w\right)$ ; this clearly defines a 
$[0,+\infty ]$-valued measure $\pi_n$ on $\rond A_{-n}=\prod_{-n\le k\le -1}\rond A$.
Besides, $\pi _1$ is a probability measure.
Indeed, because of Definition~\eqref{prolongementV} and Remark\eqref{rem:cascRem},
\[
\mu (0)+\mu (1)
=\sum _{c\in\rond C^f}\left(\casc (0c)+\casc (1c)\right)\mu\left(\alpha _cs_c\right)
=\sum _{c\in\rond C^f}\casc (c)\mu\left(\alpha _cs_c\right)
\]
which can be written
\[
\mu (0)+\mu (1)
=\sum _{\alpha s\in\rond S}\mu\left(\alpha s\right)\sum _{c\in\rond C^f,~c=\cdots [\alpha s]}\casc (c)
=\sum _{\alpha s\in\rond S}v_{\alpha s}\kappa _{\alpha s}=1
\]
the last equality coming from the assumption on $(v_{\alpha s})_{\alpha s}$.

In view of applying Kolmogorov extension theorem, the consistency condition states as follows : $\pi_{n+1}(\rond A\overline w)=\pi_n(\overline w)$ for any $w\in\rond W$ of length $n$.
This is true because
\begin{equation}
\label{eq:compatibility}
\mu(w0)+\mu(w1)=\mu(w).
\end{equation}
Indeed, for any $a\in\rond A$, since $s_w$ is internal, $s_wa$ is either internal or a context.
Furthermore,
\begin{itemize}
	\item if $s_wa\in\rond I$ then $s_{wa}=s_wa$, $\alpha_{wa}=\alpha_{w}$ and $\casc (wa)=\casc (w)$ so that 
\begin{equation}
\label{eq:pourCompatibility}
\mu (wa)=\casc(w)\sum_{c\in\rond C^f,~c=s_wa\cdots}\casc\left(\alpha_wc\right)\mu\left(\alpha_c s_c\right) ;
\end{equation}
	\item if $s_wa\in\rond C$ then denote $\kappa =s_wa$ so that $\casc (wa)=\casc (w)\casc\left(\alpha _w\kappa\right)$, $\alpha _{wa}=\alpha _\kappa$ and $s_{wa}=s_\kappa$.
	Thus, $\mu (wa)
	=\casc (wa)\sum _{c=s_\kappa\cdots}\casc\left(\alpha _\kappa c\right)\mu \left(\alpha _cs_c\right)
	=\casc (w)\casc\left(\alpha _w\kappa\right)\mu\left(\alpha _\kappa s_\kappa\right)$, which implies that~\eqref{eq:pourCompatibility} still holds, the sum being reduced to one single term since $s_wa$ is itself a context.
	\end{itemize}
Valid in all cases, Formula~\eqref{eq:pourCompatibility} easily implies Claim~\eqref{eq:compatibility}. Consequently all the $\pi_n$ are probability measures.
By Kolmogorov extension theorem, there exists a unique probability measure $\pi$ on $\rond L$ such that $\pi_{|\rond A_{-n}}=\pi_n$ for every $n$.

Furthermore, $\pi (\overline c)=0$ for any infinite context $c$.
Indeed, one has successively,
\begin{align*}
1&=\sum _{c\in\rond C^f}\pi\left(\rond L\overline c\right)+\sum _{c\in\rond C^i}\pi\left(\overline c\right)\\
&=\sum _{c\in\rond C^f}\mu (c)+\sum _{c\in\rond C^i}\pi\left(\overline c\right).
\end{align*}
Besides,
\[
\sum _{c\in\rond C^f}\mu (c)
=\sum _{c\in\rond C^f}\casc (c)\mu\left(\alpha _cs_c\right)
=\sum _{\alpha s\in\rond S}v_{\alpha s}\kappa _{\alpha s}=1
\]
so that $\sum _{c\in\rond C^i}\pi\left(\overline c\right)=0$.

The stationarity of $\pi$ follows from  the identity $\mu(0w)+\mu(1w)=\mu(w)$ for any finite word $w$. Namely :
\begin{itemize}
\item if $w\notin\rond I$, then for $a\in\rond A$, $s_{aw}=s_w$, $\alpha_{aw}=\alpha_w$ hence
\[
\mu(0w)+\mu(1w)=\left(\casc(0w)+\casc(1w)\right)\sum_{c\in\rond C^f,~ c=s_w\cdots}\casc\left(\alpha_wc\right)\mu\left(\alpha_cs_c\right) .
\]
Now, Remark~\ref{rem:cascRem} entails the claim.
\item if $w\in\rond I$, then for $a\in\rond A$, $s_{aw}=w$, $\alpha_{aw}=a$ and $\casc(aw)=1$ thus
\[
\mu(aw)=\sum_{c\in\rond C^f,~c=w\cdots}\casc (ac)\mu\left(\alpha_cs_c\right) .
\]
Using again Remark~\ref{rem:cascRem}, it comes
\begin{align*}
\mu(0w)+\mu(1w)&=\sum_{c\in\rond C^f,~c=w\cdots}\left(\casc(0c)+\casc(1c)\right)\mu(\alpha_cs_c)\\
&=\sum_{c\in\rond C^f,~c=w\cdots}\casc(c)\mu(\alpha_cs_c)\\
&=\sum_{c\in\rond C^f,~c=w\cdots}\mu(c)\\
&=\sum_{c\in\rond C^f,~c=w\cdots}\pi\left(\rond L\overline c\right)\\
&=\pi\left(\rond L\overline w\right)-\sum_{c\in\rond C^i,~c=w\cdots}\pi\left(\overline c\right) =\mu (w).
\end{align*}
\end{itemize}
Since $f(\pi )=\left( v_{\alpha s}\right) _{\alpha s}$, this concludes the proof.
\end{proof}

\subsection{Two particular cases}

\subsubsection{Finite hat, $1$ or $0$ as a context}
\label{meynTweedie}

\begin{pro}
Let $(\rond T,q)$ be a non-null probabilised context tree and $U$ the associated VLMC.
Assume that $1$ and $0^{a}$ are contexts, for some integer $a\geq 1$
(or symmetrically that $0 \in \rond C$ and $1^{a} \in \rond C$).
Then, $U$ admits a unique invariant probability measure.
\end{pro}

\begin{minipage}{0.5\textwidth}
The proof is based on Kac's theorem (see Theorem~10.2.2 together with Theorem~10.0.1 in \cite{meyn/tweedie/09}).
Indeed, for any context $c$, the cylinder $\rond L\overline{c}$ is an \emph{atom} (see \cite[Chap 5, p. 96]{meyn/tweedie/09}).
Besides, the non triviality of the $q_c$ yields the  $\lambda$-irreductibility of $(U_n)_{n\geq 0}$ (see \cite[Chap 5]{meyn/tweedie/09}), where $\lambda$ denotes the Lebesgue measure on $[0,1]$.
Due to  Kac's theorem, $(U_n)_n$ is positive recurrent if and only if $\Espe\left[\tau_1|U_0\in \rond L1\right]<\infty$, where
\[\tau_1=\inf\{n\geq 1, \quad U_n \in \rond L1\}\]
is the first return time in $\rond L1$.
Since
\end{minipage}
\begin{minipage}{0.5\textwidth}
\centering
\definecolor{qqqqff}{rgb}{0.,0.,1.}
\begin{tikzpicture}[line cap=round,line join=round,>=triangle 45,x=1.0cm,y=1.0cm]
\fill[color=qqqqff,fill=qqqqff,fill opacity=0.1] (5.08,2.84) -- (5.24,4.58) -- (6.62,2.78) --(6.62,2.78) ..controls +(-1,0.5) and +(2,-1).. (5.08,2.84);
\fill[color=qqqqff,fill=qqqqff,fill opacity=0.1] (2.94,0.24) -- (3.74,2.12) -- (4.82,0.24) --(4.82,0.24) ..controls +(-1,-0.5) and +(2,1).. (2.94,0.24);
\fill[color=qqqqff,fill=qqqqff,fill opacity=0.1] (3.92,2.4) -- (4.22,2.88) -- (4.76,2.38) -- (4.76,2.38)..controls +(0.5,-1) and +(0.5,0.2).. (3.92,2.4);
\draw (1.84,1.34)-- (4.8,5.88);
\draw (5.68,5.28)-- (4.8,5.88);
\draw (5.88,5.48) node[anchor=north west] {$q_1$};
\draw (1.6,1.32) node[anchor=north west] {$q_{0^{a}}$};
\draw (5.24,4.58)-- (4.351684937289774,5.192381626789046);
\draw (4.22,2.88)-- (3.2637086868301717,3.5236612966922234);
\draw (3.74,2.12)-- (2.7883463837784106,2.7945583048493186);
\draw [color=qqqqff] (5.24,4.58)-- (5.08,2.84);
\draw [color=qqqqff] (6.62,2.78)-- (5.24,4.58);
\draw [color=qqqqff] (3.74,2.12)-- (2.94,0.24);
\draw [color=qqqqff] (4.82,0.24)-- (3.74,2.12);
\draw [color=qqqqff] (4.22,2.88)-- (3.92,2.4);
\draw [color=qqqqff] (4.76,2.38)-- (4.22,2.88);
\draw (4,4) node[rotate=58,anchor=north west] {...};
\begin{scriptsize}
\draw [fill=qqqqff] (1.84,1.34) circle (1.5pt);
\draw [fill=qqqqff] (4.8,5.88) circle (1.5pt);
\draw [fill=qqqqff] (5.68,5.28) circle (1.5pt);
\draw [fill=qqqqff] (5.24,4.58) circle (1.5pt);
\draw [fill=qqqqff] (4.22,2.88) circle (1.5pt);
\draw [fill=qqqqff] (3.74,2.12) circle (1.5pt);
\end{scriptsize}
\end{tikzpicture}
\end{minipage}
\begin{align*}
\Espe\left[\tau_1|U_0\in \rond L1\right]&=\sum_{n=1}^{\infty}n\Proba\left(U_n \in \rond L10^{n-1}1|U_0 \in \rond L1\right)\\
&\leq \sum_{n=1}^{a}n\Proba\left(U_n \in \rond L10^{n-1}1|U_0 \in \rond L1\right)+\sum_{n=a+1}^{\infty}nq_{0^{a}}(0)^{n-a-1}q_{0^{a}}(1),
\end{align*}
and since $q_{0^{a}}(0)\neq 1$, we have $\Espe\left[\tau_1|U_0\in \rond L1\right]<\infty$. Thus, the non triviality of the $q_c$ implies that $U$ is positive recurrent and hence $U$ admits a unique invariant probability measure.


\subsubsection{Finite number of infinite branches and uniformly bounded $q_c$}
\label{paccautGallo}

\begin{pro}
Let $(\rond T,q)$ be a probabilised context tree and $U$ the associated VLMC.
Assume that

(i) $\exists\varepsilon>0$, $\forall c\in\rond C$, $\forall\alpha\in\rond A$, $\varepsilon<q_c(\alpha)<1-\varepsilon$ (strong non-nullness);

(ii) $\rond C^i$ is a finite set.

Then, $U$ admits at least one invariant probability measure.
\end{pro}

This result is a consequence of the theorem proved in \cite{gallo/paccaut/13}, stated in the framework of $g$-measures, which contains the VLMC processes. Assuming regularity conditions on the $g$-function (which plays the role of the $q_c$), a uniqueness result is also obtained.

The proof relies on a careful study of the so-called transfer operator associated to the $g$-function. The infinite contexts in the VLMC framework play the role of discontinuities of the $g$-function. When the $g$-function is continuous, it is straightforward to find a fixed point for the dual of the transfer operator. This fixed point is an invariant probability measure for the process. The result extends to the case when the $g$-function has a finite number of discontinuities.

\section{Shift-stable context trees}

\label{sec:stable}

This section deals with a subclass of context trees defined by a hypothesis put on their shape (see definition and characterizations of stable trees in Section~\ref{subsection:defstable}). For this subclass, using the notion of descent tree (see Section~\ref{subsection:descent}), the matrix $Q$ is proved to be irreducible and stochastic. In Section~\ref{subsection:existence_stable}, using Theorem~\ref{fQbij}, a necessary and sufficient condition is given for a stable tree VLMC to admit a (unique) stationary probability measure (Theorem~\ref{th:stable}). In particular, when $\rond S$ is finite, this NSC reduces to the convergence of cascade series (Theorem~\ref{cor:finite}), a rather easy to handle condition. Finally, Section~\ref{subsection:semimarkov} is dedicated to the link that can be made with the semi-Markov chains theory.

\subsection{Definitions, examples}
\label{subsection:defstable}

\begin{pro}\label{prop:defstable}
Let $\rond T$ be a context tree. The following conditions are equivalent.
\begin{enumerate}
	\item[(i)] The unlabelled tree $\rond T$ is invariant by the shift $\sigma$ : $\forall \alpha \in \rond A$, $\forall w \in \rond W$, $\alpha w \in \rond T \Longrightarrow w \in \rond T$. In an equivalent manner, $\sigma(\rond T)\subseteq\rond T$.
		\item[(ii)] If $c$ is a finite context and $\alpha \in \rond A$, then $\alpha c$ is noninternal.
		\item[(iii)] $\forall \alpha\in\rond A$, $\rond T\subset \alpha\rond T$, where $\alpha\rond T=\{\alpha w, w\in\rond T\}$.
		\item[(iv)] For any VLMC $(U_n)_n$ associated with $\rond T$, the process $\left(\lpref\left(\overline{U_n}\right)\right)_{n\in\g N}$ defines a Markov chain with state space $\rond C$.
\end{enumerate} 
\end{pro}

\begin{proof}\ 

$(i)\implies (ii)$. Take $c\in\rond C$ and $\alpha\in\rond A$. If $\alpha c\in\rond I$ then $\alpha c0\in\rond T$. The item $(i)$ implies $c0\in\rond T$, which contradicts $c\in\rond C$.

$(ii)\implies (i)$. Take $\alpha\in\rond A$ and $w$ such that $\alpha w\in\rond T$. If  $w\notin\rond T$ then there exists a finite context $c$ such that $w=cw'$ with $w'\neq\emptyset$. It comes $\alpha cw'\in\rond T$, which implies $\alpha c\in\rond I$ and this contradicts $(ii)$.

$(i)\iff(iii)$ is easy.

$(ii)\implies (iv)$ What needs to be proved is that $\lpref\left(\overline{U_{n+1}}\right)$ only depends on $U_n$ through $\lpref\left(\overline{U_n}\right)$. In other words, we shall prove that for all $s\in\rond L, \alpha \in \rond A$, $\lpref\left(\alpha\overline{s}\right)$ only depends on $s$ through $\lpref(\overline{s})$. This is clear because
$$
\lpref(\alpha\overline{s}) = \lpref\left(\alpha\lpref(\overline{s}) \right).
$$
Indeed, $\lpref(\overline{s})\in\rond C$ and (ii) implies $\alpha\lpref(\overline{s})\notin\rond I$. Therefore $\alpha\lpref(\overline{s})$ writes $cw$ with $c\in\rond C$ and $w\in\rond W$. On one hand, this entails $\lpref\left(\alpha\lpref(\overline{s}) \right)=c$. On the other hand, this means that $cw$ is a prefix of $\alpha\bar s$ thus $\lpref(\alpha\overline{s})=c$.

$(iv)\implies (ii)$
We shall prove the contrapositive. Assume there exists $c\in\rond C^f$ and $\alpha\in\rond A$ such that $\alpha c\in\rond I$. Let $s\in\rond L$ such that $\lpref(\overline{s})=c$. As $\alpha c\in\rond I$, $\lpref(\alpha\overline{s})$ is a context which has $\alpha c$ as a strict prefix. Therefore, $\lpref(\alpha\overline{s})$ does not only depend on $\lpref(\overline{s})$, but going further in the past of $s$ is needed.
\end{proof}

\begin{defi}[shift stable tree]
A context tree is \emph{shift stable}, shortened in the sequel as \emph{stable} when one of the four equivalent  conditions of Proposition~\ref{prop:defstable} is satisfied.
\end{defi}

\begin{rem}
If $\rond T$ is stable then $\sigma(\rond I)\subseteq\rond I$. Namely, if $v\in\rond I$, $v\neq\emptyset$, then  $v\alpha\in\rond T$ for any  $\alpha\in\rond A$. As $\rond T$ is stable, $\sigma(v\alpha)=\sigma(v)\alpha\in\rond T$, which implies $\sigma(v)\in\rond I$.
\end{rem}

\begin{defi}[stabilizable tree, stabilized of a tree]
A context tree is \emph{stabilizable} whenever the stable tree $\displaystyle\bigcup_{n\in\g N}\sigma^n\left(\rond T\right)$ has at most countably many infinite branches (\emph{i.e.} when the latter is again a context tree).
When this occurs, $\displaystyle\bigcup_{n\in\g N}\sigma^n\left(\rond T\right)$ is called the \emph{stabilized} of $\rond T$;
it is the smallest stable context tree containing~$\rond T$.
\end{defi}

For example, the left-comb
\begin{minipage}{35pt}
\begin{tikzpicture}[scale=0.3]
\tikzset{every leaf node/.style={draw,circle,fill},every internal node/.style={draw,circle,scale=0.01}}
\Tree [.{}
	[.{}
		[.{}
			[.{}
				[.{}  \edge[line width=2pt,dashed];\node[fill=white,draw=white]{};\edge[draw=white];\node[fill=white,draw=white]{}; {} ]
		{} ]
			{} ]
				{} ]
					{} ]
\end{tikzpicture}
\end{minipage}
is stable.
On the contrary, the bamboo blossom
\begin{minipage}{45pt}
\begin{tikzpicture}[scale=0.3]
\tikzset{every leaf node/.style={draw,circle,fill},every internal node/.style={draw,circle,scale=0.01}}
\Tree [.{}
	[.{} {} 
		[.{} [.{} {} [.{} [.{} {} [.{} \edge[line width=2pt,dashed];\node[fill=white,draw=white]{};\edge[draw=white];\node[fill=white,draw=white]{}; {} ] ] {} ] ] {} ] ]
	{}
      ]
\end{tikzpicture}
\end{minipage}
is non-stable; it is stabilizable, its stabilized being the double bamboo
\begin{minipage}{80pt}
\begin{tikzpicture}[scale=0.3]
\tikzset{every leaf node/.style={draw,circle,fill},every internal node/.style={draw,circle,scale=0.01}}
\Tree [.{}
	[.{} {} 
		[.{} [.{} {} [.{} [.{} {} [.{} \edge[line width=2pt,dashed];\node[fill=white,draw=white]{};\edge[draw=white];\node[fill=white,draw=white]{}; {} ] ] {} ] ] {} ] ] 
    \edge[draw=white];\node[fill=white,draw=white,scale=4]{}; 
		[.{} [.{} {} [.{} [.{} {} [.{} [.{} {} \edge[draw=white];\node[fill=white,draw=white]{};\edge[line width=2pt,dashed];\node[fill=white,draw=white]{};] {} ] ] {} ] ] {} ] ]
\end{tikzpicture}
\end{minipage}.

\vskip 10pt
\begin{rem}
A context tree is not necessarily stabilizable as the following examples show.

\vskip 10pt
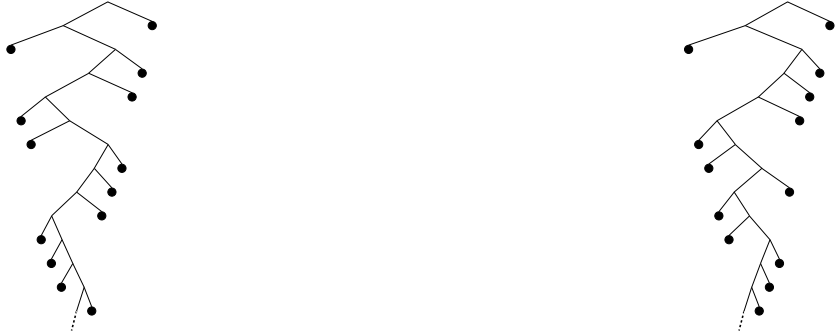
\begin{minipage}{0.47\linewidth}
\begin{center}
\begin{tikzpicture}[scale=0.3]
\tikzset{every leaf node/.style={draw,circle,fill},every internal node/.style={draw,circle,scale=0.01}}
\Tree [.{} 
		[.{} {}
		[.{} 
			[.{}
				[.{} {}
					[.{} {}
						[.{}
							[.{}
								[.{}
									[.{} {} 
										[.{} {} 
											[.{} {}
												[.{} 
													[.{} \edge[line width=2pt,dashed];\node[fill=white,draw=white]{};\edge[draw=white];\node[fill=white,draw=white]{}; ] {} 
												]
											]
										]
									]
									{} ]
								{} ]
							{} ]
					]
				]
				{} ]
			{} ]
		]
	{} ]
      
\end{tikzpicture}\\
\end{center}
This context tree consists in saturating the infinite word $010011\dots 0^k1^k\cdots$ by adding hairs.
This filament tree is stabilizable, its stabilized being the context tree having the
$\{0^l1^k0^{k+1}1^{k+1}\cdots\}$ and the $\{1^l0^k1^{k+1}0^{k+1}\cdots\}$, $k\geq 1, 0\leq l\leq k-1$ as internal nodes.
Its countably many infinite branches are the $0^k1^{\infty}$ and the $1^k0^{\infty}$, $k\geq 0$.
\end{minipage}
\hfill
\begin{minipage}{0.47\linewidth}
\begin{center}
\begin{tikzpicture}[scale=0.3]
\tikzset{every leaf node/.style={draw,circle,fill},every internal node/.style={draw,circle,scale=0.01}}
\Tree [.{} 
		[.{} {}
		[.{} 
			[.{}
				[.{}
					[.{} {}
						[.{} {}
							[.{}
								[.{} {}
									[.{} {} 
										[.{} 
											[.{}
												[.{} 
													[.{} \edge[line width=2pt,dashed];\node[fill=white,draw=white]{};\edge[draw=white];\node[fill=white,draw=white]{}; ] {} 
												]
											{} ]
										{} ]
									]
									 ]
								{} ]
							 ]
					] {}
				]
				{} ]
			{} ]
		]
	{} ]
      
\end{tikzpicture}\\
\end{center}
This context tree, denoted by $\rond T$ for a while, consists in saturating the infinite word $0100011011000001\cdots$ made of the concatenation of all finite words taken in length-alphabetical order.
It is not stabilizable.
Indeed, any finite word belongs to the smallest stable tree that contains $\rond T$, the latter having thus has uncountably many infinite branches.
\end{minipage}
\end{rem}

\begin{pro}
	\label{pro:stabilized}
	Let $(\rond T,q)$ be a stabilizable probabilised context tree and $\widehat{\rond T}$ its stabilized.
	For every context $c$ of $\widehat{\rond T}$, define $\widehat{q}_c=q_{\lpref(c)}$ where the function $\lpref$ is relative to $\rond T$. Then $(\rond T,q)$ and $(\widehat{\rond T},\widehat{q})$ define the same VLMC.
\end{pro}

\begin{minipage}{0.4\textwidth}
	\begin{proof}
		Both VLMC, as Markov processes on $\rond{L}$, have the same transition probabilities.
	\end{proof}
The example of the opposite figure illustrates the Proposition for the bamboo blossom and its stabilized tree, the double bamboo.
\end{minipage}
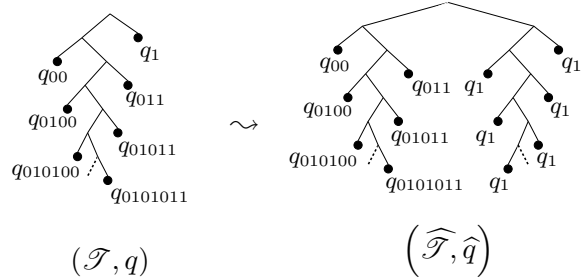
\begin{minipage}{0.6\textwidth}
\centering
\hskip 15pt
\begin{tikzpicture}[scale=0.3]
\tikzset{every leaf node/.style={draw,circle,fill},every internal node/.style={draw,circle,scale=0.001}}
\Tree [.{}
	[.{} {} 
		[.{} [.{} {} [.{} [.{} {} [.{} \edge[line width=2pt,dashed];\node[fill=white,draw=white]{};\edge[draw=white];\node[fill=white,draw=white]{}; {} ] ] {} ] ] {} ] ]
	{}
      ]
\draw (1.8,-1.8) node{{\footnotesize$q_1$}};
\draw (1.6,-3.8) node{{\footnotesize$q_{011}$}};
\draw (1.7,-6) node{{\footnotesize$q_{01011}$}};
\draw (1.8,-8.1) node{{\footnotesize$q_{0101011}$}};
\draw (-2.5,-2.8) node{{\footnotesize$q_{00}$}};
\draw (-2.5,-4.8) node{{\footnotesize$q_{0100}$}};
\draw (-2.8,-7) node{{\footnotesize$q_{010100}$}};
\draw (6,-5) node{$\leadsto$};
\draw (0,-11) node{$\left( \rond T,q\right)$};
\end{tikzpicture}
\begin{tikzpicture}[scale=0.3]
\tikzset{every leaf node/.style={draw,circle,fill},every internal node/.style={draw,circle,scale=0.001}}
\Tree [.{}
	[.{} {} 
		[.{} [.{} {} [.{} [.{} {} [.{} \edge[line width=2pt,dashed];\node[fill=white,draw=white]{};\edge[draw=white];\node[fill=white,draw=white]{}; {} ] ] {} ] ] {} ] ] 
    \edge[draw=white];\node[fill=white,draw=white,scale=4]{}; 
    \edge[draw=white];\node[fill=white,draw=white,scale=4]{}; 
		[.{} [.{} {} [.{} [.{} {} [.{} [.{} {} \edge[draw=white];\node[fill=white,draw=white]{};\edge[line width=2pt,dashed];\node[fill=white,draw=white]{};] {} ] ] {} ] ] {} ] ]
\draw (5.5,-2.8) node{{\footnotesize$q_1$}};
\draw (5,-4.8) node{{\footnotesize$q_1$}};
\draw (4.4,-7) node{{\footnotesize$q_1$}};
\draw (1.2,-3.8) node{{\footnotesize$q_1$}};
\draw (1.7,-6) node{{\footnotesize$q_1$}};
\draw (2.3,-8.1) node{{\footnotesize$q_1$}};
\draw (-5,-2.8) node{{\footnotesize$q_{00}$}};
\draw (-5.3,-4.8) node{{\footnotesize$q_{0100}$}};
\draw (-5.5,-7) node{{\footnotesize$q_{010100}$}};
\draw (-0.7,-3.8) node{{\footnotesize$q_{011}$}};
\draw (-1.2,-6) node{{\footnotesize$q_{01011}$}};
\draw (-1,-8.1) node{{\footnotesize$q_{0101011}$}};
\draw (0,-11) node{$\left( \widehat{\rond T},\widehat q\right)$};
\end{tikzpicture}
\end{minipage}

\subsection{Descent tree of a stable context tree}
\label{subsection:descent}

In the stable case, the finite contexts organize in a remarkable way: all the contexts that have the same $\alpha$-lis may be seen as the nodes of a tree called \emph{descent tree}. The labels of these nodes are read from right to left (unlike the usual case where labels are read from left to right). This situation is precised by the following lemmas. In particular, Lemma~\ref{lem:fondLem} explains how the various node types of a descent tree (with two children, one or no child) correspond to different context types having two, only one or the empty lis as a prefix.

\begin{lem}
\label{lem:contextStableTree}
Let $\left(\rond T,q\right)$ be a stable context tree.

(i) Any context $\alpha$-lis is a context.
In otherwords, $\rond S\subseteq\rond C$.

(ii) Assume that $c=\cdots [\alpha s]\in\rond C^f$. Then all $\sigma ^k(c)$, $0\leq k\leq |c|-|\alpha s|$ are also contexts.

(iii) For any $\alpha s\in\rond S$, the set $\{ c\in\rond C^f, c=\cdots [\alpha s] \}$ constitutes the nodes of a tree, defined below as a \emph{descent tree}.
\end{lem}

\begin{proof}
Let $\alpha s\in\rond S$ and let $c\in\rond C^f$ such that $c=\cdots [\alpha s]$.
Since $\rond T$ is stable, for any $k\in\g N$, the node $\sigma ^k(c)$ is internal or a context.
By maximality of $s$, this implies that the $\sigma ^k(c)$, for $0\leq k\leq |c|-|\alpha s|$, have $\alpha s$ as a suffix and are noninternal, thus contexts.
This proves (ii), thus (i) and (iii).
\end{proof}

\begin{defi}
Let $\rond T$ be a context tree and $\alpha s\in\rond S$ be a context $\alpha$-lis.
The \emph{descent tree associated with $\alpha s$} is the tree whose nodes are the finite contexts of $\rond T$ having $\alpha s$ as an $\alpha$-lis. The nodes are labelled reading from right to left
(instead of the more common other labelling that reads words from left to right). The root is $\alpha s$. This tree is denoted by $\rond D_{\alpha s}$:
\[
\rond D_{\alpha s}=\left\{ \sigma ^n(c),~c\in\rond C,~c=\cdots [\alpha s],~0\leq n\leq |c|-|\alpha s| \right\}.
\]
The \emph{saturated descent tree associated with $\alpha s$} is the descent tree of $\alpha s$ completed by the daughters of all its nodes:
\[
\overline{\rond D_{\alpha s}}=\rond D_{\alpha s}\cup\left\{ \beta c,~\beta\in\rond A,~c\in\rond D_{\alpha s}\right\}.
\]
\end{defi}
As an illustration, see the double bamboo in Example \ref{exa:dbamboo}. For the left-comb of left-comb of Section~\ref{subsubsec:pgPg}, the descent tree that corresponds to the $\alpha$-lis $10^n1$ is simply an infinite left-comb.


\begin{exa}[\bf double bamboo]
\label{exa:dbamboo}
\ \\

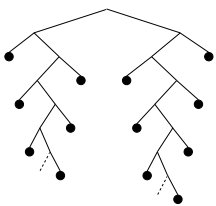
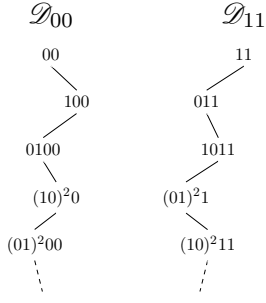
\begin{minipage}{0.3\textwidth}
\centering
\begin{tikzpicture}[scale=0.3]
\tikzset{every leaf node/.style={draw,circle,fill},every internal node/.style={draw,circle,scale=0.01}}
\Tree 
[.{} 
	[.{} 
	{} 
		[.{} 
		[.{} 
		{} 
		 [.{} 
		 [.{} 
		  {} 
		  [.{} 
		  \edge[line width=1.2pt,dashed];\node[fill=white,draw=white]{};\edge[draw=white]; 
		  \node[fill=white,draw=white]{}; {} ] ] {} ] ] {} ] 
      ] 
    \edge[draw=white];\node[fill=white,draw=white,scale=4]{}; 
 [.{} 
	[.{} 
	{} 
		[.{} 
		[.{} 
		{} 
		 [.{} 
		 [.{} 
		  {} 
		  [.{} 
		  \edge[line width=1.2pt,dashed];\node[fill=white,draw=white]{};\edge[draw=white]; 
		  \node[fill=white,draw=white]{}; {} ] ] {} ] ] {} ] ]
	{} 
      ] 
]
\end{tikzpicture}
\end{minipage}
\begin{minipage}{0.7\textwidth}
\centering
\begin{tabular}{r|l}
$\alpha$-lis $\alpha s$&contexts having $\alpha s$ as an $\alpha$-lis\\
\hline
$00$&$(01)^k00, k\geq 0$; $(10)^k0, k\geq 1$\\
$11$&$(10)^k11, k\geq 0$ ; $(01)^k1, k\geq 1$
\end{tabular}
\end{minipage}
\end{exa}

\vskip 10pt
\begin{minipage}{0.3\textwidth}
\centering
The descent trees:
\vskip 5pt
\begin{tikzpicture}[baseline,scale=0.6]
	\Tree 
	[.{$00$} 
	\edge[draw=white];\node[fill=white,draw=white]{}; 
	[.{$100$} 
	[.{$0100$} 
	\edge[draw=white];\node[fill=white,draw=white]{}; 
	[.{$(10)^20$} 
	[.{$(01)^200$} 
	\edge[draw=white];\node[fill=white,draw=white]{}; 
	\edge[dashed];\node[fill=white,draw=white]{}; 
	] 
	\edge[draw=white];\node[fill=white,draw=white]{}; ] 
	\edge[draw=white];\node[fill=white,draw=white]{}; ] 
	\edge[draw=white];\node[fill=white,draw=white]{}; ] 
	\edge[draw=white];\node[fill=white,draw=white]{}; ]
	\draw (0,1) node{$\rond D_{00}$};
\end{tikzpicture}
\begin{tikzpicture}[baseline,scale=0.6]
	\Tree 
	[.{$11$} 
	[.{$011$} 
	\edge[draw=white];\node[fill=white,draw=white]{}; 
	[.{$1011$} 
	[.{$(01)^21$} 
	\edge[draw=white];\node[fill=white,draw=white]{}; 
	[.{$(10)^211$} 
	\edge[dashed];\node[fill=white,draw=white]{}; 
	\edge[draw=white];\node[fill=white,draw=white]{}; 
	] 
	] 
	\edge[draw=white];\node[fill=white,draw=white]{}; ] 
	\edge[draw=white];\node[fill=white,draw=white]{}; ] 
	\edge[draw=white];\node[fill=white,draw=white]{}; ] 
	\draw (0,1) node{$\rond D_{11}$};
\end{tikzpicture}
\end{minipage}
\begin{minipage}{0.7\textwidth}
\centering
The saturated descent trees:
\vskip 5pt
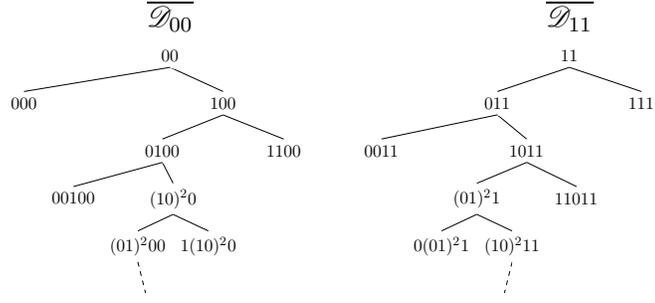
\begin{tikzpicture}[baseline,scale=0.6]
	\Tree 
	[.{$00$} 
	{$000$}
	[.{$100$} 
	[.{$0100$} 
	{$00100$}
	[.{$(10)^20$} 
	[.{$(01)^200$} 
	\edge[draw=white];\node[fill=white,draw=white]{}; 
	\edge[dashed];\node[fill=white,draw=white]{}; 
	] 
	{$1(10)^20$} ] 
	\edge[draw=white];\node[fill=white,draw=white]{}; ] 
	{$1100$} ] 
	\edge[draw=white];\node[fill=white,draw=white]{}; ]
	\draw (0,1) node{$\overline{\rond D_{00}}$};
\end{tikzpicture}
\hskip 10pt
\begin{tikzpicture}[baseline,scale=0.6]
	\Tree 
	[.{$11$} 
	[.{$011$} 
	{$0011$}
	[.{$1011$} 
	[.{$(01)^21$} 
	{$0(01)^21$}
	[.{$(10)^211$} 
	\edge[dashed];\node[fill=white,draw=white]{}; 
	\edge[draw=white];\node[fill=white,draw=white]{}; 
	] 
	] 
	{$11011$} ] 
	\edge[draw=white];\node[fill=white,draw=white]{}; ] 
	{$111$} ] 
	\draw (0,1) node{$\overline{\rond D_{11}}$};
\end{tikzpicture}
\end{minipage}

\vskip 20pt

\begin{lem}
\label{lem:fondLem}
Let $\left(\rond T,q\right)$ be a stable context tree.
Let $c\in\rond C$ be a context of $\rond T$.

\newcommand{\esp}{\hskip 12pt}
(i) Are equivalent:

\esp
(i.1) $0c\in\rond C$ and $1c\in\rond C$;

\esp
(i.2) $c$ does not admit any context lis as a prefix.

(ii) Let $\alpha\in\rond A$. Are equivalent:

\esp
(ii.1) $\alpha c\notin \rond C$ and $\overline{\alpha}c\in\rond C$;

\esp
(ii.2) there exists a unique context lis $t$ such that $c=t\cdots$, $\alpha t\in\rond C$ and $\overline{\alpha}t\notin\rond C$.

(iii) Are equivalent:

\esp
(iii.1) $0c\notin\rond C$ and $1c\notin\rond C$;

\esp
(iii.2) there exists a unique context lis $t_0$ and a unique context lis $t_1$ (that might be equal) such that:
$c=t_0\cdots=t_1\cdots$, $0t_0\in\rond C$ and $1t_1\in\rond C$.
\end{lem}

\begin{proof}
$\bullet (i.1\Rightarrow i.2)$ Assume that $t$ is a context lis such that $c=t\cdots$.
Then $0t$ or $1t$ is a context $\alpha$-lis.
Since $\rond T$ is stable, Lemma~\ref{lem:contextStableTree} implies that $0t$ or $1t$ is thus a context.
Consequently, $0c\notin\rond C$ or $1c\notin\rond C$ because two different contexts cannot be prefix of one another.

$\bullet (ii.1\Rightarrow ii.2)$
Assume (say) that $\alpha =0$, \emph{i.e.} that $0c\notin\rond C$ and $1c\in\rond C$.

(a)
Existence of $t$.
Since $\rond T$ is stable, $\rond T$ is a subtree of $\rond A\rond T$ so that $0c$, which is thus noninternal, is an external node.
Let $c'=\rpref (0c)$.
The context $c'$ is a strict prefix of $0c$ that satisfies $c'=0\cdots$.
Since $\rond T$ is stable, $\sigma (c')$ is nonexternal.
But $\sigma (c')$ cannot be a context because it is a prefix of $c$ which is a context.
Thus $\sigma (c')$ is internal.
This implies that $t:=\sigma (c')$ is the lis of $c'$ and a prefix of $c$ as well, and that $0t=c'\in\rond C$.
Finally, since $c'=0t=\rpref (0c)$, $t$ is a prefix of $c$, so that $1t$ is a prefix of the context $1c$.
Consequently, $1t\notin\rond C$ because two different contexts cannot be prefix of one another.

(b)
uniqueness of $t$.
Assume that $c=t\cdots =t'\cdots$ where $t$ and $t'$ are different context lis's such that $0t\in\rond C$ and $0t'\in\rond C$.
Then $0t$ and $0t'$ are different contexts that are both prefixes of $0c$, thus prefix of one another, which is not possible.

$\bullet (iii.1\Rightarrow iii.2)$
Assume that $0c\notin\rond C$ and $1c\notin\rond C$.
As in the proof of $(ii.1\Rightarrow ii.2)$, there is a unique context lis $t_0$
such that $c=t_0\cdots$ and $0t_0\in\rond C$.
By the same argument, there is a unique context lis $t_1$ such that $c=t_1\cdots$ and $1t_1\in\rond C$.

$\bullet $ End of the proof.

On one side, (i.1), (ii.1) and (iii.1) are disjoint cases that cover all possible situations.
On the other side, the same can be said about (i.2), (ii.2) and (iii.2), because of what follows. A context cannot be written $c=s\cdots =t\cdots$ where, for a same $\alpha\in\rond A$, $\alpha s\in\rond S$, $\alpha t\in\rond S$ and $s\neq t$.
Indeed, once again, $\alpha s$ and $\alpha t$ would be different contexts that are both prefixes of $\alpha c$, thus prefix of one another, which is not possible.
Thus, the three equivalences are proven.
\end{proof}

\begin{rem}
\label{rem:descent}
In terms of descent trees, Lemma~\ref{lem:fondLem} can be seen the following way.
A context that belongs to case (i) has valence $3$ in its descent tree
(one parent, two children, it is called a bifurcation).
A context that belongs to case (ii) has valence $2$ in its descent tree
(one parent, one child, it is called monoparental).
A context that belongs to case (iii) has valence $1$ in its descent tree (it is a leaf: one parent, no child).
\end{rem}

\subsection{Properties of Q}

\begin{defi}
A (finite or denumerable) matrix $\left( a_{r,c}\right)_{r,c}$ is said \emph{row-stochastic} whenever all its rows (are summable and) sum to $1$, \emph{i.e.}
\[
\forall r,~\sum _{c}a_{r,c}=1.
\]
\end{defi}

\begin{pro}\label{pro:stochasticity}
Let $\left(\rond T,q\right)$ be a stable probabilised context tree.
Assume that the cascade series~\eqref{convCasc} converge.
Then, the matrix $Q$ is row-stochastic.
\end{pro}

The row-stochasticity of $Q$ writes
\[
\forall \alpha s\in\rond S,~\sum _{\beta t\in\rond S}Q_{\alpha s,\beta t}=1.
\]

\begin{lem}
\label{lem:cascLem}
Let $\rond D$ be the descent tree associated with some $\alpha$-lis of a stable probabilised context tree and let $\overline{\rond D}$ be the associated saturated descent tree.

(i) For any $n\in\g N$, if $\overline{\rond D}_n=\{ w\in\overline{\rond D},~|w|\leq n\}$ denotes the $n$-th truncated tree of $\overline{\rond D}$ and $\rond L\left(\overline{\rond D}_n\right)$ the set of its leaves, then
\[
1
=\sum _{\ell\in\rond L\left(\overline{\rond D}_n\right)}\casc (\ell).
\]

(ii) Denote by $\rond I\left(\rond D\right)$ the set of internal nodes of $\rond D$ and $\rond L\left(\overline{\rond D}\right)$ the set of finite leaves of~$\overline{\rond D}$.
If $\displaystyle\sum _{\ell\in\rond I\left(\rond D\right)}\casc\left(\ell\right)$ is summable, then
\[
1
=\sum _{\ell\in\rond L\left(\overline{\rond D}\right)}\casc (\ell).
\]
\end{lem}

\begin{proof}[Proof of Lemma~\ref{lem:cascLem}]
Denote by $\rond T$ the probabilised context tree and by $\alpha s$ the root of the descent tree~$\rond D$.

\mbox{
\begin{minipage}[b][][c]{300pt}
(i) Let $\ell$ be a leaf of $\overline{\rond D}_n$ and assume that $\ell$ is an internal node of $\overline{\rond D}_{n+1}$.
In particular, $\ell\in\rond D$ which means that $\ell$ is a suffix of some context $c$ with $\alpha$-lis $\alpha s$.
Then, either $\ell =c$ is a context, or $\ell$ writes $\ell =\cdots\alpha s$ and $c$ writes $c=\cdots\beta\ell=\cdots [\alpha s]$, $\beta\in\rond A$, which prevents $\ell$ to be internal (in $\rond T$) by maximality of $s$.
In any case, $\ell$ is a noninternal node of $\rond T$.
Consequently, by Remark \ref{rem:cascRem}, the daughters $0\ell$ and $1\ell$ are leaves of $\overline{\rond D}_{n+1}$ that satisfy $\casc (\ell)=\casc (0\ell )+\casc (1\ell )$.
This proves (i) by induction on $n$
(note that the cascade of an $\alpha$-lis is always $1$).
\end{minipage}
\hskip 1cm
\begin{tikzpicture}[scale=0.8]
\Tree [.\node[fill,circle,color=black,scale=0.4]{}; 
[.\node[fill,circle,color=black,scale=0.4]{}; 
\edge[dashed];{} 
\edge[white];{} %
\edge[white];{} %
\edge[dashed];{} 
]
[.\node[fill,circle,color=black,scale=0.4]{}; 
[.\node[fill,circle,color=black,scale=0.4]{}; 
\edge[dashed];{} 
[.\node[fill,circle,color=black,scale=0.4]{}; 
[.\node[fill,circle,color=black,scale=0.4]{}; 
\edge[dashed];{} 
\edge[dotted, line width=1.5pt];{} 
]
\edge[dashed];{} 
]
]
\edge[dashed];{} 
]
]
\draw (0.4,0.2) node{$\alpha s$};
\draw (1.1,-2.8) node{$\ell$};
\draw [dashed] (-1.5,-3.17)--++(4,0);
\draw (-1.9,-3.15) node{$n$};
\end{tikzpicture} 
}

(ii) Because of (i), for any $n$,
\[
\sum _{\substack{{\ell\in\rond L\left(\overline{\rond D}\right)}\\{|\ell |\leq n}}}\casc\left(\ell\right)
=\sum _{\ell\in\rond L\left(\overline{\rond D}_n\right)}\casc (\ell)
-\sum _{\substack{{\ell\in\rond I\left(\overline{\rond D}\right)}\\{|\ell |=n}}}\casc (\ell )
=
1-\sum _{\substack{{\ell\in\rond I\left(\rond D\right)}\\{|\ell |=n}}}\casc (\ell ).
\]
Because of the summability assumption, the last sum indexed by $\ell$ tends to $0$ when $n$ tends to infinity.
Since the cascades are nonnegative numbers, this shows that $\sum _{\ell\in\rond L\left(\overline{\rond D}\right)}\casc (\ell)$ is summable and proves the result.
\end{proof}

\begin{proof}[Proof of Proposition~\ref{pro:stochasticity}]
Let $\alpha s\in\rond S$.
Name $\rond D=\rond D_{\alpha s}$ its descent tree and $\overline{\rond D}=\overline{\rond D_{\alpha s}}$ its saturated descent tree.
The assumption guarantees that the family $\left(\casc (c)\right) _{c\in\rond D}$ is summable, or equivalently that the family $\left(\casc (\ell)\right) _{\ell\in\rond I\left(\rond D\right)}$ is summable (notations of Lemma~\ref{lem:cascLem}).
Remember that $Q_{\alpha s,\beta t}$ is the sum of $\casc\left( \beta c\right)$ where $c$ runs over all contexts such that $c=\cdots [\alpha s]=t\cdots$.
Take $c\in\rond D$, which means that $c$ is a context with $\alpha$-lis $\alpha s$.
The proof is now based on Lemma~\ref{lem:fondLem}.
If $c$ belongs to case (i), it admits no context-lis as a prefix so that it never appears as a term in some $Q_{\alpha s,\beta t}$.
If $c$ belongs to case (ii), then the only $\beta t\in\rond S$ such that $c=t\cdots$ is a leaf of $\overline{\rond D}$ whose sister is an internal node of $\overline{\rond D}$.
If finally $c$ belongs to case (iii), then the contexts $0t_0\in\rond S$ and $1t_1\in\rond S$ such that $c=t_0\cdots=t_1\cdots$ give rise to two sister leaves of $\overline{\rond D}$, $0c$ and $1c$.
Putting the three cases together one sees that any word of the form $\beta c$ where $c=\cdots [\alpha s]=t\dots$ and $\beta t\in\rond S$ appears once and only once as a term of some $Q_{\alpha s,\beta t}$.
Consequently, the sum of all $Q_{\alpha s,\beta t}$ where $\beta t$ runs over $\rond S$ can be written as the sum of the cascades of all leaves of $\overline{\rond D}$.
By Lemma~\ref{lem:cascLem}, this leads to the result.
\end{proof}

\begin{rem}
\label{rem:realisation}
Any stochastic matrix with strictly positive coefficients $A=(a_{ij})_{i\geqslant 0,j\geqslant 0}$ is a $\alpha$-lis transition matrix of a non-null stable context tree. It may be realised with a left-comb of left-combs (see Example~\ref{subsubsec:pgPg}).

\vskip 5pt
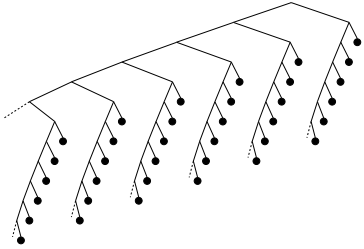
\begin{minipage}{0.4\textwidth}
\centering
\begin{tikzpicture}[scale=0.25]
		\tikzset{every leaf node/.style={draw,circle,fill},every internal node/.style={draw,circle,scale=0.01}}
\Tree [.{}
	[.{}
		[.{}
			[.{} 
					[.{} 
						[.{} \edge[dashed];\node[fill=white,draw=white]{};  [.{} 
		[.{}
			[.{}
				[.{}
					[.{} 
						[.{} \edge[dashed];\node[fill=white,draw=white]{}; {} ]
						{} ]
					{} ]
			{} ] 
		{} ]
	{} ]]
						[.{} 
		[.{}
			[.{}
				[.{}
					[.{} 
						[.{} \edge[dashed];\node[fill=white,draw=white]{}; {} ]
						{} ]
					{} ]
			{} ] 
		{} ]
	{} ]
					]
				[.{} 
		[.{}
			[.{}
				[.{}
					[.{} 
						[.{} \edge[dashed];\node[fill=white,draw=white]{}; {} ]
						{} ]
					{} ]
			{} ] 
		{} ]
	{} ]
 			]
		[.{} 
		[.{}
			[.{}
				[.{}
					[.{} 
						[.{} \edge[dashed];\node[fill=white,draw=white]{}; {} ]
						{} ]
					{} ]
			{} ] 
		{} ]
	{} ]
		] [.{} 
		[.{}
			[.{}
				[.{}
					[.{} 
						[.{} \edge[dashed];\node[fill=white,draw=white]{}; {} ]
						{} ]
					{} ]
			{} ] 
		{} ]
	{} ]
	]	 
	[.{} 
		[.{}
			[.{}
				[.{}
					[.{} 
						[.{} \edge[dashed];\node[fill=white,draw=white]{}; {} ]
						{} ]
					{} ]
			{} ] 
		{} ]
	{} ]
	]
\end{tikzpicture}
\end{minipage}
\begin{minipage}{0.5\textwidth}
The contexts are $0^{i}10^{j}1$, $i,j\geqslant 0$, the $\alpha$-lis of $0^{i}10^{j}1$ being $10^j1$.
One can check that
\[
	Q_{i,j}:=Q_{10^i1,10^j1}=\casc(10^j10^i1).
\]
\end{minipage}

\vskip 10pt
A calculation shows that if
\[
	q_{0^i10^j1}(1)=\frac{a_{ji}}{1-\sum_{k=0}^{i-1}(1-a_{jk})},
\]
then $Q_{ij}=a_{ij}$.
The question whether any stochastic matrix (with some zero coefficients) can be realised seems to be more difficult. Namely, zero coefficients in $Q$ assuming non-zero $q_c(\alpha)$ constraint the shape of the context tree.
\end{rem}

\begin{pro}
\label{Qirreductible}
Let $\left(\rond T,q\right)$ be a non-null stable probabilised context tree.
Then the matrix $Q$ is irreducible.
\end{pro}

\begin{proof}
	Let $\alpha s$ and $\beta t$ be two $\alpha$-lis. We shall show that there is a path in the matrix $Q$ from $\alpha s$ to $\beta t$. Recall that there is a path of length one from $\alpha s$ to $\alpha' s'$ if there exists a context the $\alpha$-lis of which is $\alpha s$ and which begins with the lis $s'$, $\alpha's'$ being a $\alpha$-lis. In this case, the non-nullness of $(\rond T,q)$ implies that the transition from $\alpha s$ to $\alpha' s'$ is not zero.

To illustrate the transition from $\alpha s$ to $\beta t$, let us use the concatenated word
$$\beta t\alpha s=\beta t_qt_{q-1}\dots t_1\alpha s.$$
Let us start from $\alpha s$ and add letters from $t$ to the left. As $\alpha s\in\rond C$ and $\rond T$ is stable, when we add letters to the left, we get either contexts or strictly external words. Let us distinguish the two following cases:
\begin{itemize}
\item Either we always get contexts. In this case, since no internal node is obtained by adding the letters of $t$, $t\alpha s$ is a context with  $\alpha$-lis $\alpha s$ . Moreover, as $\beta t$ is a context, $\beta t\alpha s\notin\rond T$ which implies by  Lemma~\ref{lem:fondLem} that the context $t\alpha s$ begins with the unique lis $t$. As $\beta t$ is an $\alpha$-lis, this is the case of a one-step transition form $\alpha s$ to $\beta t$.
\item Or let $k$ be the least integer such that $t_k\dots t_1\alpha s\not\in\rond T$. Firstly $t_{k-1}\dots t_1\alpha s$ is a context with $\alpha$-lis $\alpha s$ (same argument as above). Secondly, by Lemma~\ref{lem:contextStableTree}, the context $t_{k-1}\dots t_1\alpha s$ begins with a unique lis, denoted by $u$ and $t_ku$ is an $\alpha$-lis. Now the aim is to have a transition with the whole of $t$. Therefore, one needs to check that $t_ku$ contains the whole end of $t$. If it were not the case, $u$ would write $t_{k-1}\dots t_l$ with $l>1$ and $t_k\dots t_l$ would be a context and a suffix of the internal node $\beta t_q\dots t_l$. This can't happen in a stable tree. We have found a transition from $\alpha s$ to $t_k u=t_k\dots t_1u'$. Now we go on adding letters from $t$ to the left and we again have the dichotomy described in the two items of this proof. By repeating this a finite number of times, we get the result.
\end{itemize}
\vskip -15pt
\end{proof}

\subsection{Stationary measure for the VLMC vs recurrence of $Q$}
\label{subsection:existence_stable}

The following result relates the existence and the uniqueness of a stationary probability measure of a VLMC to the recurrence of $Q$.
Let us recall the definition of recurrence and state a necessary and sufficient condition to get a (unique) invariant probability measure for stable trees.

\begin{defi}
Let $A=(a_{ij})$ be a stochastic irreducible countable matrix.
Denote by $a_{ij}^{(k)}$ the $(i,j)$-th entry of the (stochastic) matrix $A^k$.
The matrix $A$ is \emph{recurrent} whenever there exists $i$ such that
\[
\sum_{k=1}^{\infty}a_{ii}^{(k)}=1.
\]
\end{defi}
Any stochastic irreducible countable matrix may be viewed as the transition matrix of an irreducible Markov chain with countable state space. 
The recurrence means that there is a state $i$ (and this is true for every state because of irreducibility) for which the first return time is a.s. finite.
When in addition the expectation of the return times are finite, the matrix is classically called \emph{positive recurrent}.

\begin{theo}
\label{th:stable}
Let $(\rond T,q)$ be a non-null probabilised context tree.
Assume that $\rond T$ is stable.
Then, the following assertions are equivalent.
\begin{enumerate}
\item The VLMC associated to $(\rond T,q)$ has a unique stationary probability measure
\item The VLMC associated to $(\rond T,q)$ has at least a stationary probability measure
\item The three following conditions are satisfied:
\begin{itemize}
\item[$(c_1)$] the cascade series~\eqref{convCasc} converge
\item[$(c_2)$] $Q$ is recurrent
\item[$(c_3)$] $\sum _{\alpha s\in\rond S} v_{\alpha s}\kappa_{\alpha s}<\infty$ where $\g R \left( v_{\alpha s}\right) _{\alpha s}$ is the unique line of left-fixed vectors of $Q$.
\end{itemize}
\end{enumerate}
\end{theo}

\begin{proof}\ \\
({\it 3.$\implies$1.}) Since $Q$ is recurrent and irreducible, there exists a unique line $\g Rv$ of left-fixed vectors for~$Q$, where $v=\left( v_{\alpha s}\right) _{\alpha s\in\rond S}$
(see for example \cite[Theorem 5.4]{seneta/06}).
Theorem~\ref{fQbij}(ii) coupled with the assumption on the series $\sum _{\alpha s\in\rond S} v_{\alpha s}\kappa_{\alpha s}$ entails directly the existence and uniqueness of a stationary probability measure.

({\it 2.$\implies$3.})
If there exists a stationary probability measure, then Theorem~\ref{fQbij}(i) asserts that the cascade series converge and that $Q$ admits at least one direction of left-fixed vectors $\g Rv$ such that $\sum _{\alpha s\in\rond S} v_{\alpha s}\kappa_{\alpha s}<\infty$.
Besides, every $\kappa _{\alpha s}$ is greater than $1$.
Indeed, the cascade of any $\alpha$-lis is $1$ and, in the stable case, any $\alpha$-lis is a context.
Thus, $v$ is summable and $Q$ is positive recurrent
(see for instance~\cite[Corollary of Theorem 5.5]{seneta/06}).
Since it is irreducible (Proposition~\ref{Qirreductible}), it admits a unique direction of left-fixed vectors $\g Rv$, proving ({\it 3.}) by Theorem~\ref{fQbij}.
\end{proof}

\begin{rem}
Actually, as shown in the end of the proof, when $Q$ is recurrent and when the series $\sum _{\alpha s\in\rond S} v_{\alpha s}\kappa_{\alpha s}$ converge, then $Q$ is positive recurrent.
\end{rem}

\begin{rem}
One can find a VLMC defined by a stable tree such that the cascade series converge and the matrix $Q$ is transient.

To build such an example, recall that, by Remark \ref{rem:realisation}, any stochastic matrix with strictly positive coefficients can be realized as the matrix $Q$ of a stable tree (take for example a left-comb of left-combs). The matrix $A=(a_{ij})_{i\geqslant 1,j\geqslant 1}$ defined by
\begin{itemize}
\item $a_{i,i+1}=1-\frac{1}{(i+1)^2}$ for all $i\geqslant 1$,
\item $a_{ij}=\frac{1}{(i+1)^22^{j-1}}$ if $j\geqslant i+2$,
\item $a_{ij}=\frac{1}{(i+1)^22^{i+1-j}}$ if $j\leqslant i$.
\end{itemize}
is stochastic and transient.
Indeed, if one associates a Markov chain to the stochastic matrix $A$ and if one denotes by $T_1$ the return time to the first state,
\[
\Proba(T_1=\infty)\geqslant\prod_{i\geqslant 1}a_{i,i+1}\geqslant\prod_{i\geqslant 2}\left(1-\frac{1}{i^2}\right) =\frac12.
\]
Consequently, $A$ does not have any nonzero left-fixed vector.

Consider now the VLMC defined by a left-comb of left-combs (see~\ref{subsubsec:pgPg}) probabilised in the unique way such that $Q_{10^q1,10^p1}=a_{q,p}$ for every $(p,q)$.
A simple computation shows that the series of cascade converge (geometrically).
Simultaneously, since $Q$ is transient, Theorem~\ref{th:stable} shows that the VLMC admits no stationary probability measure.
\end{rem}

\begin{rem}
\label{rem:contreexemple}
Assume that the context tree is stable and that the cascade series~\eqref{convCasc} converge.

If the VLMC admits a stationary probability measure, then, as already seen, every left-fixed vector of $Q$ is summable (the family of its coordinates is summable).

The reciprocal implication is false:
the left-fixed vectors of $Q$ may be summable while no finite measure is stationary for the VLMC, because condition $(c_3)$ in Theorem~\ref{th:stable} is not satisfied.
One can find in Section~\ref{subsubsec:pgpgSommable} such an example with a left-comb of left-combs.
\end{rem}

\begin{rem}
If one removes the stability assumption, the $\kappa _{\alpha s}$ may not be bounded below, so that the argument that shows the summability of left-fixed vectors of $Q$ fails.
Indeed, with a non-stable context tree having an infinite $\rond S$, one may have
\[
\inf _{\alpha s\in\rond S}\kappa _{\alpha s}=0.
\]
Such an example is developed in~\ref{subsec:housse}.
\end{rem}

Notice that Theorem~\ref{th:stable} also provides results for non-stable trees as the following corollary shows, using notations of Proposition~\ref{pro:stabilized}.

\begin{cor}
	\label{cor:stabilized}
	Let $(\rond T,q)$ be a non-null probabilised context tree. Suppose that $\rond T$ is stabilizable and denote by $\widehat{\rond T}$ its stabilized. If $(\widehat{\rond T},\widehat{q})$ satisfies the conditions of Theorem~\ref{th:stable}, then the VLMC associated with $(\rond T,q)$ admits a unique invariant probability measure. If not, it does not admit any invariant probability measure.
	In particular, a VLMC associated to such a context tree $(\rond T,q)$ never admits several stationary probability measures.
\end{cor}

\begin{proof}
	This is a consequence of Proposition~\ref{pro:stabilized} and Theorem~\ref{th:stable}.
\end{proof}

\subsection{Existence and uniqueness when $\rond S$ is finite and $\rond T$ stable}

When the matrix $Q$ is finite, stochasticity and irreducibility are sufficient to get a unique left-fixed vector, therefore

\begin{theo}[finite number of $\alpha$-lis]
\label{cor:finite}
Let $(\rond T,q)$ be a non-null probabilised context tree.
Assume that $\rond T$ is stable and that $\#\rond S<\infty$.
Then (i) and (ii) are equivalent.

(i) The VLMC associated to $(\rond T,q)$ has a unique stationary probability measure.

(ii) The cascade series~\eqref{convCasc} converge. 
\end{theo}

\begin{proof}
(i)$\implies$(ii) is contained in Theorem~\ref{fQbij}(i).
Assume reciprocally that the cascade series converge.
Since $Q$ is stochastic, irreducible and finite dimensional, it admits a unique direction of left-fixed vectors, so that Theorem~\ref{fQbij}(ii) allows us to conclude.
\end{proof}

To see how this theorem applies, see both examples in Section~\ref{sec:pgPd}.
The first one -- the left-comb of right-combs --  is spectacularly simple ($\#\rond S=1$).
The second one is a straightforward application of Theorem~\ref{cor:finite}.

\subsection{The $\alpha$-lis process as a semi-Markov chain}
\label{subsection:semimarkov}

For this section, semi-Markov chains can be comprehended thanks to \cite{barbu/limnios/08}. 

Remember that for a VLMC $(U_n)$ defined on a \emph{stable} context tree, by (iv) in Definition \ref{subsection:defstable}, the process $(C_n)_{n\geq0}$ defined by
\[
C_n = \lpref\left(\overline{U_n}\right)
\]
defines a Markov chain with state space $\rond C$, the set of contexts. Moreover, the process $(Z_n)_{n\geq0}$ of $\alpha$-lis of contexts, where
\[Z_n=\alpha_{C_n}s_{C_n},\]
is a semi-Markov chain with state space $\rond S$, as detailed in the following.
\begin{defi}
\label{def:semiMarkov}
Let $(J_n,S_n)_{n\geq0}$ be a Markov chain with state space $\rond S\times \g N$ 
such that $S_0 = 0$ and $(S_n)$ is increasing.
The semi-Markov chain associated with $(J,S)$ is the $\rond S$-valued process $(Z_n)_{n\geq0}$ defined by $Z_0 = J_0$ and
\[
\forall k \hbox{ such that } S_n\leq k < S_{n+1}, \hskip 5mm Z_k = J_n.
\]
In otherwords, the $S_n$ are jump times and $Z_k$ stagnates at a same state between two successive jump times. 
\end{defi}
In a VLMC $(U_n)$ defined on a \emph{stable} context tree, both processes $(C_n)$ and $(Z_n)$ evolve as follows: start with $Z_0=\alpha s$ which is the $\alpha$-lis of $C_0$. 
As time goes by, $U_n$ grows by addition of letters
and $C_n$ goes down in the descent tree of $\alpha s$ (the length of $C_n$ increases and $Z_n$ remains equal to $\alpha s$) until there is $n_0\geq 0$ and $\beta\in\rond A$ such that $\beta C_{n_0}\not\in\rond C$. Since a descent tree is not the complete tree (if not, the context tree would also be the complete tree, which is not allowed for a VLMC), $n_0$ is finite and lemma \ref{lem:fondLem} implies the existence of a context lis $t$, prefix of $C_{n_0}$ such that  $\beta t\in \rond S$. Therefore, at the next time step, with nonzero probability $q_{C_{n_0}}(\beta)$, one has $C_{n_0+1}=Z_{n_0+1}=\beta t$, the process $C_n$ has gone up to the root of the descent tree $\rond D_{\beta t}$. With probability $q_{C_{n_0}}(\beta)$, at time $n_0+1$, the process $(Z_n)$ ``changes'' state; in this case, 
$n_0+1$ is a jump time. 
Remark 
that one may have that $n_0+1$ is a jump time \emph{and} $Z_{n_0+1} = Z_{n_0}$ (it is the case when the context process goes back to the root of the descent tree it belongs to).

More generally, the evolution of $C_n$ and $Z_n$ shows that jump times may be defined as follows (with $S_0=0$)~: for $n\geq 1$,
\[S_n=\min\{k>S_{n-1},C_k=Z_k\}.\]
In particular, $C_{S_n} = Z_{S_n}$. At all the other times $k$, $Z_k$ is a strict suffix of $C_k$. The jump times may be equivalently defined as
\[S_n=\min\{k>S_{n-1},|C_k|\le|C_{k-1}|\}.\]
For $n\geq 0$, let 
\[J_n=Z_{S_n}\]
be the state of the $\alpha$-lis process at the $n$-th jump, so that $S_{n+1}-S_{n}$ is the sojourn time in the state~$J_{n}$. For $k\geq 1$, if $N(k)=\max\{n, S_n\leqslant k\}$ is the number of jump times in the interval $[1,k]$, one also have $J_{N(k)}=Z_k$.
\begin{pro}
\label{pro:semiMarkov}
Let $(U_n)$ be a VLMC on a stable context tree and assume that the cascade series~\eqref{convCasc} converge. Let  $(J_n,S_n)$ be as defined above. Then the sojourn times $S_{n+1}-S_{n}$ are a.s. finite, $(J_n,S_n)$ is a Markov chain on 
$\rond S\times \g N$, $(S_n)$ is increasing and $(Z_n)$ is a semi-Markov chain associated with $(J,S)$.
\end{pro}
\begin{proof}

Let 
\begin{align*}
q_{\alpha s,\beta t}(k)
&=\proba(J_{n+1}=\beta t,S_{n+1}=S_n+k | J_0,\dots,J_{n-1},J_n=\alpha s,S_0,\dots,S_n) \\
&=\proba(C_{S_n+1}\in\rond D_{\alpha s}\setminus \{\alpha s\},\dots,C_{S_n+k-1}\in\rond D_{\alpha s} \setminus \{\alpha s\} ,C_{S_n+k}=\beta t|C_{S_n}=\alpha s)\\
&= \sum _{\substack{c=t\cdots [\alpha s]\\|c|-|\alpha s|=k }}
\casc\left(\beta c\right), 
\end{align*}
so that $q_{\alpha s,\beta t}(k) =\proba(J_{n+1}=\beta t,S_{n+1}=S_n+k | J_n=\alpha s) $. This computation is illustrated in the example below, Remark \ref{rem:semiMarkov}.

Notice also that $\sum_{k\geq 1}q_{\alpha s,\beta t}(k)=Q_{\alpha s,\beta t}$. Therefore, the semi-Markov kernel property of $q_{\alpha s,\beta t}(k)$, namely
$\sum_{k\geq 1}\sum_{\beta t\in \rond S}q_{\alpha s,\beta t}(k) = 1$,
 follows straightforwardly from the stochasticity of $Q$, see Proposition~\ref{pro:stochasticity}.

Moreover, the stochasticity of $Q$ provides the a.s. finiteness of $S_{n+1}-S_n$. Indeed, for any $\alpha s\in\rond S$, 
\begin{align*}
\sum_{k\geq 1}\Proba\left(S_{n+1}-S_n = k | J_n=\alpha s  \right)&= \sum_{\beta t\in \rond S}\sum_{k\geq 1}\Proba\left(J_{n+1}=\beta t, S_{n+1}-S_n = k | J_n=\alpha s  \right)\\
&= \sum_{\beta t\in \rond S}\sum_{k\geq 1}  \sum _{\substack{c=t\cdots [\alpha s]\\|c|-|\alpha s|=k }}
\casc\left(\beta c\right) = \sum_{\beta t\in \rond S} Q_{\alpha s,\beta t} = 1
\end{align*}
\end{proof}

\begin{rem}
\label{rem:semiMarkov}
The semi-Markov chain contains less information than the chain $(U_n)$. 
To illustrate this, here is an example with a finite context tree.

\vskip 5pt
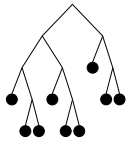
\begin{minipage}{0.4\textwidth}
\centering
\begin{tikzpicture}[scale=0.4]
		\tikzset{every leaf node/.style={draw,circle,fill},every internal node/.style={draw,circle,scale=0.01}}
\Tree [.{} [.{} [.{} {} [.{} {} {} ] ] [.{} {} [.{} {} {} ] ] ] [.{} {} [.{} {} {} ] ] 
	]
\end{tikzpicture}
\end{minipage}
\begin{minipage}{0.6\textwidth}
\centering
\begin{tabular}{r|l}
$\alpha$-lis $\alpha s$&contexts having $\alpha s$ as an $\alpha$-lis  \\
\hline
10&10,010,110,0010,0110\\
000&000\\
111&111,0111\\
0011&0011
\end{tabular}
\end{minipage}
\vskip 5pt
In this example, 0010 and 0110 are two contexts of same length, with the same $\alpha$-lis 10 and beginning by the same lis 0. Hence if we know that $J_n=10$, $S_{n+1}-S_{n}=3$ and $J_{n+1}=10$, there are two possibilities to reconstruct the VLMC $(U_n)$. With the notations of the proof above, there are two cascade terms in $q_{10,10}(3) $:

\vskip 5pt
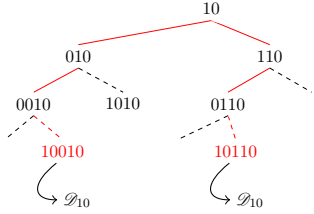
\begin{minipage}{0.32\textwidth}
\begin{tikzpicture}[baseline,scale=0.6]
	\Tree [.{$10$} \edge[draw=red];[.{$010$} \edge[draw=red];[.{$0010$} \edge[dashed];{${\color{white} 00010}$} \edge[draw=red,dashed];[.\node(10010)[]{${\color{red} 10010}$}; \edge[draw=white];{} \edge[draw=white];\node(gauche)[]{$\rond D_{10}$}; ] ] \edge[dashed];{$1010$} ] \edge[draw=red];[.{$110$} \edge[draw=red];[.{$0110$} \edge[dashed];\node[fill=white,draw=white]{${\color{white} 00110}$}; \edge[draw=red,dashed];[.\node(10110)[]{${\color{red} 10110}$}; \edge[draw=white];{} \edge[draw=white];\node(droite)[fill=white]{$\rond D_{10}$}; ] \edge[draw=white];\node[fill=white,draw=white]{}; ] \edge[dashed];\node[fill=white,draw=white]{${\color{white} 1110}$}; ] \edge[draw=white];\node[fill=white,draw=white]{}; ]
	\draw[->] (10110)..controls +(south west:1) and +(west:1)..(droite);
\draw[->] (10010)..controls +(south west:1) and +(west:1)..(gauche);
\end{tikzpicture}
\end{minipage}
\begin{minipage}{0.68\textwidth}
\begin{align*}
q_{10,10}(3) &= \proba\left(C_{S_n+1} = 010, C_{S_n+2} = 0010,C_{S_n+3} = 10010 | C_{S_n} =10\right)\\
& \ \ + \proba\left(C_{S_n+1} = 110, C_{S_n+2} = 0110,C_{S_n+3} = 10110 | C_{S_n} =10\right)\\
&= q_{10}(0) q_{010}(0) q_{0010}(1) + q_{10}(1) q_{110}(0) q_{0110}(1)\\
&= \casc (10010) + \casc (10110).
\end{align*}
\end{minipage}
\end{rem}

\section{Miscellaneous examples, bestiary}
\label{sec:examples}

\subsection{Example with $\rond S$ finite and $\rond C^i$ infinite}
\label{subsubsec:SfiniCIinfini}

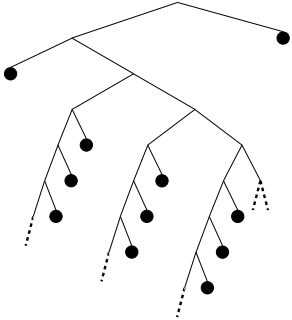
\begin{minipage}{0.4\textwidth}
\centering
\begin{tikzpicture}[scale=0.45]
\tikzset{every leaf node/.style={draw,circle,fill},every internal node/.style={draw,circle,scale=0.01}}
	\Tree [.{}
			[.{} {}
				[.{}
					[.{}
						[.{}
							[.{}
								[.{} \edge[line width=2pt,dashed];\node[fill=white,draw=white]{};\edge[draw=white];\node[fill=white,draw=white]{}; ] {}
							] {}
						] {}
					]
					[.{}
						[.{}
							[.{}
								[.{}
									[.{} \edge[line width=2pt,dashed];\node[fill=white,draw=white]{};\edge[draw=white];\node[fill=white,draw=white]{}; ] {}
								] {}
							] {}
						]
						[.{}
							[.{}
								[.{}
									[.{}
										[.{} \edge[line width=2pt,dashed];\node[fill=white,draw=white]{};\edge[draw=white];\node[fill=white,draw=white]{}; ] {}
									] {}
								] {}
							]
							[.{} \edge[line width=2pt,dashed];\node[fill=white,draw=white]{};\edge[line width=2pt,dashed];\node[fill=white,draw=white]{}; ]
						]
					]
				 ]
				]
				{}
		  ]

\end{tikzpicture}
\end{minipage}
\begin{minipage}{0.6\textwidth}
The finite contexts of this context tree are $00$, $1$, and the $01^p0^q1$, $p,q\geq 1$.
The infinite branches are the $01^p0^\infty$, $p\geq 1$, so that $\rond C^i$ is infinite.
There are four context $\alpha$-lis: $1$, $00$, $001$ and $101$, as the following array shows.
Note that there are only three context lis: $\emptyset$, $0$ and $01$.
This tree is non-stable.

\begin{center}
\begin{tabular}{r|l}
$\alpha s\in\rond S$&contexts having $\alpha s$ as an $\alpha$-lis\\
\hline
$1$&$1$\\
$00$&$00$\\
$001$&$01^p0^q1$, $p\geq 1$, $q\geq 2$\\
$101$&$01^p01$, $p\geq 1$
\end{tabular}
\end{center}
\end{minipage}

\vskip 40pt
The cascade series converge as soon as $q_1(1)\neq 1$ and $q_{00}(0)\neq 1$, since in this case,
\[\kappa_{1}=\kappa_{00}=\kappa_{001}=\kappa_{101}=1.\]
Let
\[A=\sum_{k\geq 0}q_{01^k01}(0)q_1(0)q_1(1)^{k}\]
\[B=\sum_{k\geq 0,l\geq 0}q_{01^k0^l1}(1)q_1(0)q_{00}(1)q_{1}(1)^{k}q_{00}(0)^{l}.\]
The matrix $Q$ writes as follows:
\[
Q=\left(\begin{array}{cccc}
q_1(1)&0&0&0\\
q_{00}(1)&q_{00}(0)&0&0\\
B&1-B&1-B&B\\
1-A&A&A&1-A
\end{array}\right).
\]
A simple computation leads to the unique direction of left-fixed vectors of $Q$
\[
	\left((A+B)q_{00}(1),Aq_{1}(0),Aq_1(0)q_{00}(1),Bq_1(0)q_{00}(1)\right)
\]
and Theorem~\ref{fQbij}\textit{(ii)} gives existence and uniqueness of a probability stationary measure for the VLMC associated to this probabilised context tree.
This application of Theorem~\ref{fQbij} is an alternative argument to Section~\ref{meynTweedie} this tree is an special case of.

\subsection{A non-stable tree: the brush}
\label{3lis}

This example provides an application of Theorem~\ref{fQbij} that is not covered by particular cases of Sections~\ref{meynTweedie} and~\ref{paccautGallo}.

\vskip 10pt
\begin{minipage}{0.4\textwidth}
\centering
\begin{tikzpicture}[scale=0.35]
\tikzset{
every leaf node/.style={draw,circle,fill},
every internal node/.style={draw,circle,scale=0.01}
}
\Tree
[.{} 
[.{} 
[.{} 
[.{} 
[.{} 
[.{} 
[.{} 
[.{} 
[.{} 
\edge[line width=2pt,dashed];\node[fill=white,draw=white]{};
{} 
]
{} 
]
{} 
]
{} 
]
{} 
]
{} 
]
{} 
]
[.{} 
[.{} 
[.{} 
[.{} 
[.{} 
[.{} 
[.{} 
\edge[line width=2pt,dashed];\node[fill=white,draw=white]{};
{} 
]
{} 
]
{} 
]
{} 
]
{} 
]
{} 
]
[.{} 
[.{} 
[.{} 
[.{} 
[.{} 
[.{} 
\edge[line width=2pt,dashed];\node[fill=white,draw=white]{};
{} 
]
{} 
]
{} 
]
{} 
]
{} 
]
[.{} 
[.{} 
[.{} 
[.{} 
[.{} 
\edge[line width=2pt,dashed];\node[fill=white,draw=white]{};
{} 
]
{} 
]
{} 
]
{} 
]
[.{} 
[.{} 
[.{} 
[.{} 
\edge[line width=2pt,dashed];\node[fill=white,draw=white]{};
{} 
]
{} 
]
{} 
]
[.{} 
[.{} 
[.{} 
\edge[line width=2pt,dashed];\node[fill=white,draw=white]{};
{} 
]
{} 
]
[.{} 
\edge[line width=2pt,dashed];\node[fill=white,draw=white]{};
\edge[line width=2pt,dashed];\node[fill=white,draw=white]{};
]
]
]
]
]
]
]
{} 
]
\end{tikzpicture}

\end{minipage}
\begin{minipage}{0.6\textwidth}
The finite contexts of this non-stable tree are $1$ and the $01^p0^q1$, $p\geq 0$, $q\geq 1$.
There are infinitely many infinite branches, namely the $01^p0^\infty$, $p\geq 0$.
There are only three $\alpha$-lis, as summed up in the following array.
\begin{center}
\begin{tabular}{r|l}
$\alpha$-lis $\alpha s$&contexts having $\alpha s$ as an $\alpha$-lis\\
\hline
$1$&$1$\\
$001$&$0^q1$ and $01^p0^q1$, $p\geq 1$, $q\geq 2$\\
$101$&$01^p01$, $p\geq 1$
\end{tabular}
\end{center}
\end{minipage}
\vskip 5pt
Compute the cascade series: $\kappa _1=1$, and $\kappa _{101}=1$ as soon as $q_1(1)\neq 1$.
If one denotes
\[
c_q=\casc (0^q1)=\prod _{k=2}^{q-1}q_{0^k1}(0)
\]
for every $q\geq 2$ (with $c_2=1$), then $\kappa _{001}=1+\sum _{q\geq 2}c_q$ as soon as this series converges.
Finally, under the above hypothesis of the convergence of cascade series, let
\[
A=\sum _{p\geq 1,~q\geq 2}\casc (101^p0^q1)
{\rm ~~and~~}
B=\sum _{p\geq 1}\casc (101^p01)
\]
-- these numbers are easily expressed in terms of the $q_c$.
With these notations, one gets
\[
Q=\begin{pmatrix}
q_1(1)&0&0\\
1+A&1-A&A\\
B&1-B&B
\end{pmatrix}.
\]
A simple glance to this matrix shows that $1$ is its Perron-Frobenius eigenvalue, and that all its left-fixed vectors are proportional to $\left( A+1-B,(1-B)q_1(0),Aq_1(0)\right)$.
Under the convergence of cascade series, the corresponding VLMC admits a unique stationary probability measure.

\subsection{Example with $\rond S$ infinite and $\rond C^i$ finite}
\label{subsubsec:SinfiniCIfini}

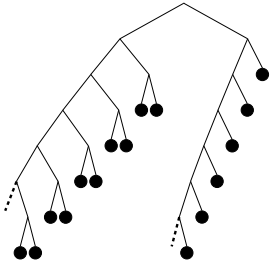
\begin{minipage}{0.4\textwidth}
\centering
\begin{tikzpicture}[scale=0.45]
		\tikzset{every leaf node/.style={draw,circle,fill},every internal node/.style={draw,circle,scale=0.01}}
\Tree [.{}
	[.{}
		[.{}
			[.{} 
					[.{} 
						[.{} \edge[line width=2pt,dashed];\node[fill=white,draw=white]{}; [.{} {} {} ] ]
						[.{} {} {} ]
					]
				[.{} {} {} ]
 			]
		[.{} {} {} ]
		] [.{} {} {} ]
	]	 
	[.{} 
		[.{}
			[.{}
				[.{}
					[.{} 
						[.{} \edge[line width=2pt,dashed];\node[fill=white,draw=white]{}; {} ]
						{} ]
					{} ]
			{} ] 
		{} ]
	{} ]
	]
\end{tikzpicture}
\end{minipage}
\begin{minipage}{0.6\textwidth}
The finite contexts of this tree are the $0^q10$, $0^q11$ and $10^p1$, $p\geq 0$, $q\geq 1$ while the infinite ones are $\rond C^i=\left\{ 0^\infty ,10^\infty\right\}$.
There are infinitely many context $\alpha$-lis, as made precise by the array.
This tree is stable.

\begin{center}
\begin{tabular}{r|l}
$\alpha s\in\rond S$&contexts having $\alpha s$ as an $\alpha$-lis\\
\hline
$11$&$0^q11$, $q\geq 1$\\
$010$&$0^q10$, $q\geq 1$\\
$10^p1$, $p\geq 1$&$10^p1$
\end{tabular}
\end{center}
\end{minipage}

\subsection{The left-comb of left-combs}
\label{subsubsec:pgPg}

\subsubsection{Definition and notations}

\begin{minipage}{0.6\textwidth}
\centering
\begin{tikzpicture}[scale=0.35]
		\tikzset{every leaf node/.style={draw,circle,fill},every internal node/.style={draw,circle,scale=0.01}}
\Tree [.{}
	[.{}
		[.{}
			[.{} 
					[.{} 
						[.{} \edge[line width=2pt,dashed];\node[fill=white,draw=white]{};  [.{} 
		[.{}
			[.{}
				[.{}
					[.{} 
						[.{} \edge[line width=2pt,dashed];\node[fill=white,draw=white]{}; {} ]
						{} ]
					{} ]
			{} ] 
		{} ]
	{} ]]
						[.{} 
		[.{}
			[.{}
				[.{}
					[.{} 
						[.{} \edge[line width=2pt,dashed];\node[fill=white,draw=white]{}; {} ]
						{} ]
					{} ]
			{} ] 
		{} ]
	{} ]
					]
				[.{} 
		[.{}
			[.{}
				[.{}
					[.{} 
						[.{} \edge[line width=2pt,dashed];\node[fill=white,draw=white]{}; {} ]
						{} ]
					{} ]
			{} ] 
		{} ]
	{} ]
 			]
		[.{} 
		[.{}
			[.{}
				[.{}
					[.{} 
						[.{} \edge[line width=2pt,dashed];\node[fill=white,draw=white]{}; {} ]
						{} ]
					{} ]
			{} ] 
		{} ]
	{} ]
		] [.{} 
		[.{}
			[.{}
				[.{}
					[.{} 
						[.{} \edge[line width=2pt,dashed];\node[fill=white,draw=white]{}; {} ]
						{} ]
					{} ]
			{} ] 
		{} ]
	{} ]
	]	 
	[.{} 
		[.{}
			[.{}
				[.{}
					[.{} 
						[.{} \edge[line width=2pt,dashed];\node[fill=white,draw=white]{}; {} ]
						{} ]
					{} ]
			{} ] 
		{} ]
	{} ]
	]
\end{tikzpicture}
\end{minipage}
\begin{minipage}{0.4\textwidth}
The \emph{left-comb of left-combs} is the context tree as drawn on the left:
the finite contexts are the $0^p10^q1$, $p,q\geq 0$.
Remark in passing that, for any corresponding VLMC, the transition probabilities of the Markov process on $\rond L$ depend only on the largest suffix of the form $0^q10^p$ of the current left-infinite sequence $U_n=\cdots 0^q10^p$
($p$ being possibly infinite).
\end{minipage}

\vskip 10pt
A left-comb of left-combs is a stable context tree.
Its has infinitely many infinite branches, namely $0^\infty$ and the $0^p10^\infty$, $p\geq 0$.
For any $p,q\geq 0$, the $\alpha$-lis of $0^p10^q1$ is $10^q1$.
In particular, the set $\rond S$ of $\alpha$-lis of contexts is countably infinite.
In this case, for any $q\geq 0$, the set of contexts having $10^q1$ as an $\alpha$-lis is also countably infinite.

Probabilise this context tree by a family $(q_c)_c$ of probability measures on $\{ 0,1\}$ and denote, for every $q,p\geq 0$,
\[
c_{q,p}=\casc (0^p10^q1)
=\prod _{0\leq k\leq p-1}q_{0^k10^q1}(0).
\]
The convergence of cascade series is equivalent to the finiteness of
\[
\kappa _{10^q1}=\sum _{p\geq 0}c_{q,p},
~\forall q\geq 0.
\]
The square matrix $Q$ is infinite, defined by
$Q_{10^q1,10^p1}=\casc (10^p10^q1)=c_{q,p}-c_{q,p+1}$
for all $p,q\geq 0$.
It has always finite entries, even if one cascade series diverges.
One sees immediately that $Q$ is line-stochastic if, and only if $c_{q,p}$ tends to $0$ as $p$ tends to infinity, for all $q\geq 0$.

\subsubsection{Stationarity and summability of left-fixed vectors of $Q$}
\label{subsubsec:pgpgSommable}

This paragraph is devoted to an example of (stable) VLMC that satisfies that the following properties:

- the cascade series converge;

- the VLCM admits no stationary probability measure;

- every left-fixed vector of $Q$ is summable.

The context tree of the example is a left-comb of left-combs with the above notations.

\vskip 5pt
Let $v_p=\frac 1{p+1}-\frac 1{p+2}$ and $R_p=\sum _{q\geq p}v_q=\frac 1{p+1}$ for every $p\geq 0$
(more generally, on can build a similar counter-example based on any positive sequence $(v_p)_p$ such that $\sum _{p\geq 0}v_p=1$ and $\sum pv_p$ diverge).
Define $S$ by
\[
S(x)=\sum _{q\geq 0}v_qx^{\frac 1{q+1}}.
\]
The series is normally convergent on the real interval $[0,1]$ so that $S$ is continuous on $[0,1]$ and satisfies $S(0)=0$ and $S(1)=1$.
Furthermore, $S$ is derivable and increasing on $[0,1]$ since the derived series converges normally on any compact subset of $]0,1]$.
Finally, $S(x)\geq v_qx^{\frac 1{q+1}}$ on $[0,1]$ for every $q\geq 0$.
Consequently, for every $t>0$, there exists $C_t>0$ such that
\begin{equation}
\label{S-1plate}
S^{-1}(x)\leq C_tx^t
\end{equation}
for every $x\in [0,1]$.

\vskip 5pt
Take now the probabilised left-comb of left-combs defined by the relations
\[
\forall q,p\geq 0,~c_{q,p}=S^{-1}\left( R_p\right)^{\frac 1{q+1}} .
\]
Note that these equations fully define the corresponding VLCM because the probabilities $q_{0^p10^q1}$ are characterized by these $c_{q,p}$ \emph{via} the equalities $q_{0^p10^q1}(0)=c_{q,p+1}/c_{q,p}$.
The definition of $S$ implies that $\sum _{q\geq 0}v_qc_{q,p}=R_p$ for every $p\geq 0$, which precisely means that $v=vQ$
(the row-vector $v$ is a left-fixed vector for $Q$).
Besides, for any $q\geq 0$, applying~\eqref{S-1plate} for $t=2(q+1)$ leads to inequalities
\[
\forall p\geq 0,~c_{q,p}\leq C_{2(q+1)}\left(\frac 1{p+1}\right)^2.
\]

Thus, the sequences $(v_q)$ and $(c_{q,p})$ satisfy the following properties.
\begin{enumerate}
\item $\forall q\geq 0, \sum_p c_{q,p}<\infty$,
\item $\forall p\geq 0, \sum_{q\geq 0} v_qc_{q,p}=\sum_{q\geq p}v_q$,
\item $\sum_q v_q<\infty$,
\item $\sum_{q,p\geq 0}v_qc_{q,p}=+\infty$.
\end{enumerate}
In terms of the VLMC, with general notations of Section~\ref{sec:general}, these properties translate into:
\begin{enumerate}
\item the cascade series converge,
\item $(v_{\alpha s})_{\alpha s\in\rond S}$ is a left-fixed vector for $Q$,
\item $\sum_{\alpha s\in\rond S} v_{\alpha s}<\infty$,
\item there exists a unique stationary positive measure $\pi$ on $\rond L$ such that $\pi\left(\rond L\overline{\alpha s}\right)=v_{\alpha s}$ for every $\alpha s\in\rond S$. The measure $\pi$ is not finite.
\end{enumerate}
The existence and uniqueness of the measure $\pi$ in item 4. can be shown by simple adaptation of the proof of Theorem~\ref{fQbij}, the total mass of $\pi$ being
\[\pi(\rond L)=\sum_{\alpha s\in\rond S}v_{\alpha s}\kappa_{\alpha s}=\sum_{\substack{{\alpha s\in\rond S}\\{c\in\rond C,~c=\cdots[\alpha s]}}}\casc(c)v_{\alpha s}
=\sum_{q,p\geq 0}v_qc_{q,p}.
\]

\subsection{Tree of small kappas}
\label{subsec:housse}

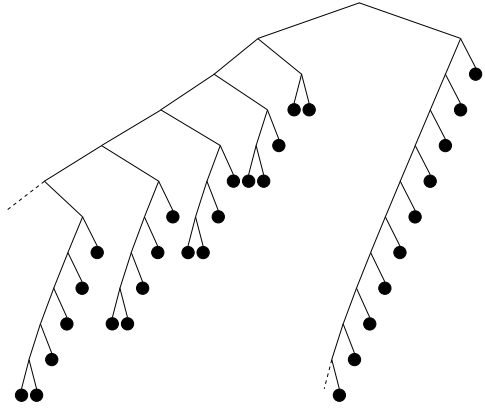
\begin{minipage}{0.5\textwidth}
\begin{tikzpicture}[scale=0.45]
\tikzset{
every leaf node/.style={draw,circle,fill},
every internal node/.style={draw,circle,scale=0.01}
}
\Tree [.{} 
[.{} 
[.{} 
[.{} 
[.{} 
[.{} 
\edge[dashed];\node[fill=white,draw=white]{};
[.{} 
[.{} 
[.{} 
[.{} 
[.{} 
{} 
{} 
]
{} 
]
{} 
]
{} 
]
{} 
]
]
[.{} 
[.{} 
[.{} 
[.{} 
{} 
{} 
]
{} 
]
{} 
]
{} 
]
]
[.{} 
[.{} 
[.{} 
{} 
{} 
]
{} 
]
{} 
]
]
[.{} 
[.{} 
{} 
{} 
]
{} 
]
]
[.{} 
{} 
{} 
]
]
[.{} 
[.{} 
[.{} 
[.{} 
[.{} 
[.{} 
[.{} 
[.{} 
[.{} 
[.{} 
\edge[dashed];\node[fill=white,draw=white]{};
{} 
]
{} 
]
{} 
]
{} 
]
{} 
]
{} 
]
{} 
]
{} 
]
{} 
]
{} 
]
]
\end{tikzpicture}
\end{minipage}
\begin{minipage}{0.5\textwidth}
The tree of small kappas is the context tree as drawn on the left.
Its finite contexts are the following ones:

$\star$ $0^m10^k1$, $m\geq 1$, $0\leq k\leq m-1$;

$\star$ $0^m10^m$, $m\geq 1$;

$\star$ $10^m1$, $m\geq 0$.

This context tree gets two infinite branches, namely $0^\infty$ and $10^\infty$.
It is non-stable.
There are infinitely many context $\alpha$-lis, as summed up in the following array.

\begin{center}
\begin{tabular}{r|l}
$\alpha$-lis $\alpha s$&contexts having $\alpha s$ as an $\alpha$-lis \\
\hline
$11$&$0^m11,~m\geq 0$\\
$10^k1,~ k\geq 1$&$10^k1$ and $0^m10^k1,~m\geq k+1$\\
$10^m,~m\geq 1$&$0^m10^m$\\
\end{tabular}
\end{center}
\end{minipage}

\vskip 5pt
When the context tree is probabilised, the convergence of cascade series is equivalent to the convergence of the series
\[
\kappa _{11}=\sum _{m\geq 0}\prod _{j=0}^{m-1}q_{0^j11}(0)
{\rm ~~and~~}
\kappa _{10^k1}=1+\sum _{m\geq k+1}\prod _{j=0}^{m-1}q_{0^j10^k1}(0),
~\forall k\geq 1.
\]
The remaining $\kappa _{\alpha s}$ are defined by sums of one sole term, namely, for all $m\geq 1$,
\[
\kappa _{10^m}=\casc (0^m10^m)=\prod _{j=1}^{m-1}q_{0^j10^j}(0).
\]
In particular, the sequence $\left(\kappa _{10^m}\right) _{m}$ is not bounded below by any positive number as soon as the infinite product diverges to $0$.
\Nfin


\subsection{Variations on the left-comb of right-combs}
\label{sec:pgPd}

This Section produces two examples of stable context trees that give rise to a direct application of Theorem~\ref{cor:finite}.
The first one, named left-comb of right-combs, is particularly simple because if has only one $\alpha$-lis of contexts.
The left-comb of right-combs augmented by a cherry stem, a variation of the former one, gets four $\alpha$-lis of contexts.
Because of Theorem~\ref{cor:finite}, both corresponding VLMC have a (unique) stationary probability measure if, and only if their cascade series converge.

\vskip 10pt
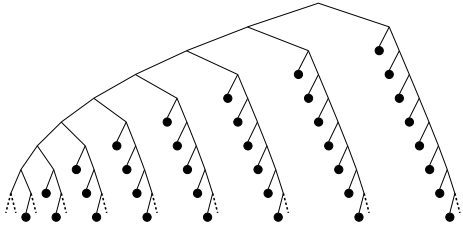
\begin{minipage}{0.4\textwidth}
\centering
\begin{tikzpicture}[scale=0.3]
\tikzset{
every leaf node/.style={draw,circle,fill},
every internal node/.style={draw,circle,scale=0.01}
}
\Tree [.{} 
[.{} 
[.{} 
[.{} 
[.{} 
[.{} 
[.{} 
[.{} 
[.{} 
\edge[line width=2pt,dashed];\node[fill=white,draw=white]{};
\edge[line width=2pt,dashed];\node[fill=white,draw=white]{};
]
[.{} 
{} 
\edge[line width=2pt,dashed];\node[fill=white,draw=white]{};
]
]
[.{} 
{} 
[.{} 
{} 
\edge[line width=2pt,dashed];\node[fill=white,draw=white]{};
]
]
]
[.{} 
{} 
[.{} 
{} 
[.{} 
{} 
\edge[line width=2pt,dashed];\node[fill=white,draw=white]{};
]
]
]
]
[.{} 
{} 
[.{} 
{} 
[.{} 
{} 
[.{} 
{} 
\edge[line width=2pt,dashed];\node[fill=white,draw=white]{};
]
]
]
]
]
[.{} 
{} 
[.{} 
{} 
[.{} 
{} 
[.{} 
{} 
[.{} 
{} 
\edge[line width=2pt,dashed];\node[fill=white,draw=white]{};
]
]
]
]
]
]
[.{} 
{} 
[.{} 
{} 
[.{} 
{} 
[.{} 
{} 
[.{} 
{} 
[.{} 
{} 
\edge[line width=2pt,dashed];\node[fill=white,draw=white]{};
]
]
]
]
]
]
]
[.{} 
{} 
[.{} 
{} 
[.{} 
{} 
[.{} 
{} 
[.{} 
{} 
[.{} 
{} 
[.{} 
{} 
\edge[line width=2pt,dashed];\node[fill=white,draw=white]{};
]
]
]
]
]
]
]
]
[.{} 
{} 
[.{} 
{} 
[.{} 
{} 
[.{} 
{} 
[.{} 
{} 
[.{} 
{} 
[.{} 
{} 
[.{} 
{} 
\edge[line width=2pt,dashed];\node[fill=white,draw=white]{};
]
]
]
]
]%
]
]
]
]
\end{tikzpicture}

\end{minipage}
\begin{minipage}{0.6\textwidth}
The finite contexts of the left-comb of right-combs (drawn on the left) are the $0^p1^q0$, $p\geq 0$, $q\geq 1$.
It has infinitely many infinite branches, namely the $0^p1^\infty$, $p\geq 0$.
This context tree is stable and all finite contexts have $01$ as an $\alpha$-lis.
The matrix~$Q$, which is thus $1$-dimensional, is reduced to $(1)$.
The convergence of the unique cascade series consists in the summability of the double sum
\[
\sum _{p\geq 0,q\geq 1}\prod _{j=0}^{p-1}q_{0^j1^q0}(0)\prod _{k=1}^{q-1}q_{1^k0}(1).
\]
Note that the transition probabilities of this Markov chain depend only of the largest suffix of the form $1^q0^p$ of the current left-infinite sequence $U_n=\cdots 1^q0^p$ ($q$ being possibly infinite).
\end{minipage}

\vskip 10pt
\begin{minipage}{0.6\textwidth}
The left-comb of right-combs with a cherry stem consists in simply replacing the context $01$ of the preceding tree by the cherries $100$ and $101$.
The tree is still stable and it has four context $\alpha$-lis, as resumed in the array.
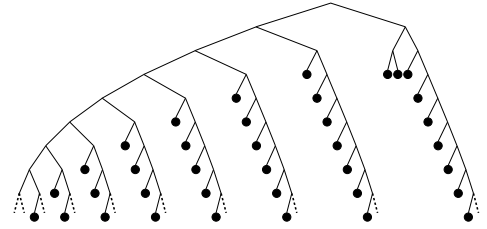
\begin{center}
\begin{tabular}{r|l}
$\alpha$-lis $\alpha s$&contexts having $\alpha s$ as an $\alpha$-lis\\
\hline
$100$&$100$\\
$101$&$101$\\
$010$&$0^p10$, $p\geq 1$\\
$110$&$0^p1^q0$, $p\geq 0$, $q\geq 2$
\end{tabular}
\end{center}
\end{minipage}
\begin{minipage}{0.4\textwidth}
\centering
\begin{tikzpicture}[scale=0.3]
\tikzset{
every leaf node/.style={draw,circle,fill},
every internal node/.style={draw,circle,scale=0.01}
}
\Tree [.{} 
[.{} 
[.{} 
[.{} 
[.{} 
[.{} 
[.{} 
[.{} 
[.{} 
\edge[line width=2pt,dashed];\node[fill=white,draw=white]{};
\edge[line width=2pt,dashed];\node[fill=white,draw=white]{};
]
[.{} 
{} 
\edge[line width=2pt,dashed];\node[fill=white,draw=white]{};
]
]
[.{} 
{} 
[.{} 
{} 
\edge[line width=2pt,dashed];\node[fill=white,draw=white]{};
]
]
]
[.{} 
{} 
[.{} 
{} 
[.{} 
{} 
\edge[line width=2pt,dashed];\node[fill=white,draw=white]{};
]
]
]
]
[.{} 
{} 
[.{} 
{} 
[.{} 
{} 
[.{} 
{} 
\edge[line width=2pt,dashed];\node[fill=white,draw=white]{};
]
]
]
]
]
[.{} 
{} 
[.{} 
{} 
[.{} 
{} 
[.{} 
{} 
[.{} 
{} 
\edge[line width=2pt,dashed];\node[fill=white,draw=white]{};
]
]
]
]
]
]
[.{} 
{} 
[.{} 
{} 
[.{} 
{} 
[.{} 
{} 
[.{} 
{} 
[.{} 
{} 
\edge[line width=2pt,dashed];\node[fill=white,draw=white]{};
]
]
]
]
]
]
]
[.{} 
{} 
[.{} 
{} 
[.{} 
{} 
[.{} 
{} 
[.{} 
{} 
[.{} 
{} 
[.{} 
{} 
\edge[line width=2pt,dashed];\node[fill=white,draw=white]{};
]
]
]
]
]
]
]
]
[.{} 
[.{} 
{} 
{} 
]
[.{} 
{} 
[.{} 
{} 
[.{} 
{} 
[.{} 
{} 
[.{} 
{} 
[.{} 
{} 
[.{} 
{} 
\edge[line width=2pt,dashed];\node[fill=white,draw=white]{};
]
]
]
]
]%
]
]
]
]
\end{tikzpicture}

\end{minipage}

\vskip 5pt
In this last example, the convergence of cascade series is equivalent to the finiteness of both sums
\[
\kappa _{010}=\sum _{p\geq 1}\prod _{k=1}^{p-1}q_{0^k10}(0)
{\rm ~~and~~}
\kappa _{110}=\sum _{p\geq 1}\prod _{j=0}^{p-1}q_{0^j1^q0}(0)\prod _{k=2}^{q-1}q_{0^k1^q0}(1).
\]

\section{More about the non-stable case}
\label{sec:towards}

Staying in the framework of non-nullness of the $q_c(\alpha)$  for all $c\in\rond C$ and $\alpha\in\rond A$ in order to avoid degenerate situations, we think that the non-stable case may be much more investigated, thanks to the following tracks.

Namely, for \emph{totally non-stable} context trees, as defined below, we claim the following conjecture.

\begin{defi}
\label{def:totally}
A context tree is \emph{totally non-stable}, when the set of infinite branches $\rond C^i$ has no shift-invariant subset.
\end{defi}

Among examples in Section \ref{sec:examples}, Example \ref{subsubsec:SfiniCIinfini}, is totally non-stable, though Examples \ref{3lis} and \ref{subsec:housse} are non-stable but not totally non-stable (they have an infinite comb among their infinite branches).

\begin{conj}
Let $(\rond T,q)$ be a non-null probabilised context tree.
If $\rond T$ is totally non-stable, then there exists a unique probability stationary measure for the VLMC associated to $(\rond T, q)$.
\end{conj}

Using Proposition~\ref{pro:stabilized} as in Corollary~\ref{cor:stabilized}, a first step in this direction consists in proving the following weaker conjecture.

\begin{conj}
Let $(\rond T,q)$ be a non-null probabilised context tree.
If $\rond T$ is totally non-stable and stabilizable, then its stabilized satifies the conditions of Theorem~\ref{th:stable}.
\end{conj}
In particular, we believe that the conditions of convergence of cascade series that are required by the existence of a stationary probability measure (see (i) in Theorem~\ref{fQbij}) are due to the existence of shift-stable subsets of $\rond C^i$.
In otherwords, for a totally non-stable stabilizable context tree, the convergence of cascade series of the stabilized automatically holds because the general terms of these series decay exponentially fast.
The very simple bamboo blossom example below comforts this impression.

\vskip 5pt
The bamboo blossom is the context tree
\begin{minipage}{26pt}
\begin{tikzpicture}[scale=0.2]
\tikzset{every leaf node/.style={draw,circle,fill},every internal node/.style={draw,circle,scale=0.01}}
\Tree [.{}
	[.{} {} 
		[.{} [.{} {} [.{} [.{} {} [.{} \edge[line width=2pt,dashed];\node[fill=white,draw=white]{};\edge[draw=white];\node[fill=white,draw=white]{}; {} ] ] {} ] ] {} ] ]
		{}
      ]
\end{tikzpicture}
\end{minipage}.
One can refer to~\cite{cenac/chauvin/paccaut/pouyanne/12} to retrieve a necessary and sufficient condition for existence and uniqueness of a stationary probability measure.
The bamboo blossom is stabilizable, its stabilized being the double bamboo
\begin{minipage}{48pt}
\begin{tikzpicture}[scale=0.2]
\tikzset{every leaf node/.style={draw,circle,fill},every internal node/.style={draw,circle,scale=0.01}}
\Tree [.{}
	[.{} {} 
		[.{} [.{} {} [.{} [.{} {} [.{} \edge[line width=2pt,dashed];\node[fill=white,draw=white]{};\edge[draw=white];\node[fill=white,draw=white]{}; {} ] ] {} ] ] {} ] ]
		\edge[draw=white];\node[fill=white,draw=white]{};
		[.{} [.{} {} [.{} [.{} {} [.{} [.{} {} \edge[draw=white];\node[fill=white,draw=white]{};\edge[line width=2pt,dashed];\node[fill=white,draw=white]{};] {} ] ] {} ] ] {} ] ]
\end{tikzpicture}
\end{minipage}.
The double bamboo bossom has two $\alpha$-lis: $00$ and $11$.
The corresponding sums of cascade series are respectively
\begin{align*}
	\kappa_{00}&=\sum_{n\geq 0}\prod_{k=1}^{n-1}q_{(10)^k0}(0)\prod_{k=0}^{n-1}q_{(01)^k00(1)}+\sum_{n\geq 1}\prod_{k=1}^{n-2}q_{(10)^k0}(0)\prod_{k=0}^{n-1}q_{(01)^k00(1)}\\
	\kappa_{11}&=\sum_{n\geq 0}\prod_{k=0}^{n-1}q_{(10)^k11}(0)\prod_{k=1}^{n-1}q_{(01)^k1(1)}+\sum_{n\geq 1}\prod_{k=0}^{n-1}q_{(10)^k11}(0)\prod_{k=1}^{n-2}q_{(01)^k1(1)}.
\end{align*}
If one probabilises the double bamboo in the sense of Proposition~\ref{pro:stabilized}, then $q_{(10)^k0}=q_1$ for all $k\geq 1$ and $q_{(10)^k11}=q_1$ for all $k\geq 0$.
The sums of cascade series become thus
\begin{align*}
	\kappa_{00}&=\sum_{n\geq 0}q_1(0)^{n-1}\prod_{k=0}^{n-1}q_{(01)^k00(1)}+\sum_{n\geq 1}q_1(0)^{n-2}\prod_{k=0}^{n-1}q_{(01)^k00(1)}\\
	\kappa_{11}&=\sum_{n\geq 0}q_1(0)^{n}\prod_{k=1}^{n-1}q_{(01)^k1(1)}+\sum_{n\geq 1}q_1(0)^{n}\prod_{k=1}^{n-1}q_{(01)^k1(1)},
\end{align*}
converging as soon as $q_{1}(0)<1$.
This is the only condition to ensure existence and uniqueness of the stationary probability measure, as it was stated in \cite{cenac/chauvin/paccaut/pouyanne/12}. Note that the set of infinite branches of the bamboo blossom has no shift-invariant subset.


\bibliographystyle{plainnat} 
\bibliography{paccaut}

\end{document}